\documentclass[10pt]{amsart}
\usepackage[all,cmtip]{xy}
\usepackage[english]{babel}
\usepackage{graphicx}
\usepackage{amssymb}
\def\today{\number\day\space\ifcase\month\or   January\or February\or
   March\or April\or May\or June\or   July\or August\or September\or
   October\or November\or December\fi\   \number\year}

\theoremstyle{definition}
\newtheorem{thm}{Theorem}[section]
\newtheorem{lemma}[thm]{Lemma}
\newtheorem{prop}[thm]{Proposition}
\newtheorem{df}[thm]{Definition}
\newtheorem{cor}[thm]{Corollary}
\newtheorem{cnj}[thm]{Conjecture}

\newtheorem{rem}[thm]{Remark}

\newtheorem{eg}[thm]{Example}

\newtheorem{pbm}[thm]{Problem}

\newtheorem{qst}[thm]{Question}

\newcommand{\beq}{\begin{equation}}
\newcommand{\eeq}{\end{equation}}
\newcommand{\beqa}{\begin{eqnarray*}}
\newcommand{\eeqa}{\end{eqnarray*}}
\newcommand{\bal}{\begin{align*}}
\newcommand{\eal}{\end{align*}}
\newcommand{\bi}{\begin{itemize}}
\newcommand{\ei}{\end{itemize}}
\newcommand{\be}{\begin{enumerate}}
\newcommand{\ee}{\end{enumerate}}
\newcommand{\bigast}{\scalebox{1.5}{$\ast$}}

\newcommand{\ep}{\varepsilon}
\newcommand{\zt}{\zeta}

\newcommand{\Z}{{\mathbb{Z}}}

\newcommand{\C}{{\mathbb{C}}}
\newcommand{\N}{{\mathbb{N}}}

\newcommand{\K}{{\mathcal{K}}}

\newcommand{\B}{{\mathcal{B}}}
\newcommand{\U}{{\mathcal{U}}}

\newcommand{\T}{{\mathbb{T}}}

\newcommand{\Ot}{{\mathcal{O}_2}}
\newcommand{\OI}{{\mathcal{O}_{\I}}}

\newcommand{\id}{{\mathrm{id}}}

\newcommand{\diag}{{\mathrm{diag}}}

\newcommand{\Aut}{{\mathrm{Aut}}}

\newcommand{\Ad}{{\mathrm{Ad}}}

\newcommand{\dimRok}{{\mathrm{dim}_\mathrm{Rok}}}
\newcommand{\cdimRok}{{\mathrm{dim}_\mathrm{Rok}^\mathrm{c}}}

\newcommand{\tfae}{the following are equivalent}
\newcommand{\ifo}{if and only if }

\newcommand{\ca}{$C^*$-algebra}

\newcommand{\uca}{unital $C^*$-algebra}

\newcommand{\hm}{homomorphism}

\newcommand{\fd}{finite dimensional}

\newcommand{\Rp}{Rokhlin property}

\newcommand{\Rdim}{Rokhlin dimension}

\newcommand{\I}{\infty}

\title[]{Rokhlin dimension for compact group actions}

\author{Eusebio Gardella}

\date{\today}

\thanks{This material is based upon work supported by the
  US National Science Foundation through Grant
DMS-1101742; the Johnson Fellowship at the University of Oregon;
and by the the Deutsche Forschungsgemeinschaft
(SFB 878). These sources of financial support are
gratefully acknowledged.}

\address{Department of Mathematics, University  of Oregon,
      Eugene OR 97403-1222, USA.}

\email[]{gardella@uoregon.edu}

\address{Fields Institute for Research in Mathematical Sciences,
222 College Street, Toronto ON M5T 3J1, Canada.}

\email[]{egardell@fields.utoronto.ca}

\begin{document}

\begin{abstract} We introduce and systematically study the notion of Rokhlin dimension (with and without commuting
towers)
for compact group actions on $C^*$-algebras. This notion generalizes the one introduced by Hirshberg, Winter
and Zacharias for finite groups, and contains the Rokhlin property as the zero dimensional case.
We show, by means of an example, that commuting towers cannot always be arranged, even in the absence of $K$-theoretic
obstructions. For a compact Lie group
action on a compact Hausdorff space, freeness is equivalent to finite Rokhlin dimension of the induced action. 
We compare the notion of finite Rokhlin dimension to other existing definitions of noncommutative freeness
for compact group actions. We obtain further $K$-theoretic obstructions to having an action of
a non-finite compact Lie group with finite Rokhlin dimension with commuting towers, and use them to confirm
a conjecture of Phillips. 
\end{abstract}

\maketitle

\tableofcontents

\section{Introduction}

The study of group actions on \ca s, as well as their associated crossed products, has been the object of
very intensive research since the early beginnings of the operator algebra theory. Both in the von Neumann algebra
and in the \ca\ case, crossed products have provided a great deal of highly nontrivial examples via a construction
that combines the dynamical properties of the action together with the structural properties of the underlying
algebra. \\
\indent A crucial result in the context of measurable dynamics is the Rokhlin Lemma, which asserts that an
aperiodic measure preserving action of $\Z$ can be ``approximated'', in a suitable sense, by finite cyclic shifts.
Its reformulation in terms of outer automorphisms of (commutative) von Neumann algebras using partitions of unity
consisting of orthogonal projections has led to a number of versions of the Rokhlin property, both in the von Neumann
algebra and in the \ca\ case. See \cite{kishimoto autom UHF} and \cite{izumi automorphisms} for integer actions, and see
\cite{herman jones}, \cite{izumi finite group actions I}, \cite{izumi finite group actions II} and \cite{osaka phillips}
for the finite group case.\\
\indent On the side of topological dynamics, the notion of freeness of a group action is central. Recall that an
action of a group $G$ on a space $X$ is said to be \emph{free} if no non-trivial group element
of $G$ acts with fixed points. Viewing \ca s as noncommutative topological spaces, it is natural to look for
generalizations of the concept of freeness to the case of group actions on \ca s. It turns out that there is not
a single version of noncommutative freeness. The book \cite{phillips equivariant k-theory book} provides a detailed
presentation and comparison of a number of them, mainly for compact Lie groups, including (locally) discrete
$K$-theory, (total) $K$-freeness, (hereditary) saturation, and others. The more recent survey article
\cite{phillips survey}, which considers mostly finite groups, incorporates other notions that have been intensively
studied in the past two decades: the Rokhlin property and the tracial Rokhlin property. We refer the reader to the
introduction of \cite{phillips survey} for a motivation of the study of free actions on \ca s.\\
\indent More recently, Hirshberg, Winter and Zacharias introduced in \cite{HWZ} the notion of finite Rokhlin dimension
for finite group actions (as well as automorphisms), as a generalization of the Rokhlin property. This more general
notion has the Rokhlin property as its zero dimensional case, and moreover has the advantage of not requiring the
existence of projections in the underlying algebra. Finite Rokhlin dimension is in particular much more common than
the Rokhlin property. \\
\indent The paper \cite{HP} consists of a further study of finite Rokhlin dimension, where the authors extend the
notion to the non-unital setting, and also derive some $K$-theoretical obstructions in the commuting tower version.
These obstructions are used to show that no non-trivial finite group acts with finite Rokhlin dimension with commuting
towers on either the Jiang-Su algebra
$\mathcal{Z}$, or the Cuntz algebra $\mathcal{O}_\I$. There, Hirshberg and Phillips introduce the notion of $X$-Rokhlin
property for an action of a compact Lie group $G$ and a compact free $G$-space $X$. They show that when $G$ is finite and $X$
is finite dimensional, the $X$-Rokhlin property is equivalent to having finite Rokhlin dimension with commuting
towers in the sense of Hirshberg-Winter-Zacharias. \\
\indent Our work develops the concept of Rokhlin dimension for compact group actions on \uca s, generalizing the
case of finite group actions of \cite{HWZ}, the Rokhlin property in the compact group case as in \cite{hirshberg winter},
and including the $X$-Rokhlin property from \cite{HP}, which is shown to be
equivalent to finite Rokhlin dimension with commuting towers in our sense, at least for Lie groups. The starting point of this project was the
simple observation that if $\alpha\colon\T\to\Aut(A)$ is an action of the circle with the
Rokhlin property on a unital \ca\ $A$, and if $n$ is a positive integer, then the restriction of $\alpha$ to the finite subgroup $\Z_n\subseteq\T$
has Rokhlin dimension at most one.
Theorem \ref{restrictions} can be regarded as a significant generalization of this fact. \\
\ \\
\indent This paper is organized as follows. In Section 2, we establish the notation that will be used throughout the
paper, and briefly recall the necessary preliminary definitions and results. In Section 3, we introduce and systematically
study the notion of Rokhlin dimension for compact group actions on \uca s. In particular, we show that finite Rokhlin dimension is
preserved under a number of constructions, namely tensor products, direct limits, passage to quotients by invariant
ideals, and restriction to closed subgroups of finite codimension. In Section 4, we compare the notion of having finite Rokhlin dimension
(mostly in the commuting tower case) with other existing forms of freeness of group actions on \ca s. We show
that in the commutative case, finite Rokhlin dimension is equivalent to freeness of the action on the maximal ideal space; see Theorem
\ref{free action on spaces}. Moreover, for a compact Lie group action, the formulation with commuting towers is equivalent
to the $X$-Rokhlin property introduced in \cite{HP}; see Theorem \ref{strongly free and finite Rdim}. We apply this to
deduce that actions with finite Rokhlin dimension with commuting towers have discrete $K$-theory and are totally $K$-free.
Theorem \ref{K-thy restrictions for K-freeness} establishes $K$-theoretic obstructions that are complementary to the ones
established in \cite{HP}. We use this result to confirm a Conjecture of Phillips from \cite{phillips equivariant k-theory book};
see Remark \ref{rem: conjecture}. Our results in fact show that Phillips' conjecture holds for a class of \ca s which
is much larger than the class of AF-algebras, without assuming that the action is specified by the way it is constructed.\\
\indent We show in Example \ref{eg: commuting matters} that commuting towers cannot always be arranged, even at the cost
of considering additional towers, and even in the absence of $K$-theoretic obstructions. Indeed, this example (originally constructed by Izumi in \cite{izumi finite group actions I})
has Rokhlin dimension 1 with noncommuting towers, and infinite Rokhlin dimension with commuting towers. Theorem \ref{thm: summary}
collects and summarizes the known implications between the notions considered in Section 4, and it also references
counterexamples that show that no other implications hold in full generality. Finally,
in Section 5, we give some indication of possible directions for future work, and raise some natural questions related
to our findings. \\
\indent Some applications of the results in this paper will appear in \cite{gardellaRegProp} and \cite{GHS}.\\
\ \\
%\indent In Section 5, we explore the structure of the crossed product and fixed point algebra of an action of a compact
%group with finite Rokhlin dimension, in relation to their decomposition rank (Theorem \ref{dr}), their nuclear
%dimension (Theorem \ref{nuc dim}), and $\mathcal{Z}$-absorption (Theorem \ref{thm: cp Z-abs}). \\
\indent All groups will be second countable. (Many statements are true in greater generality, but second-countable
groups suffice for our purposes.) By a theorem of Birkoff-Kakutani (Theorem 1.22 in \cite{montgomery zippin}), a topological
group is metrizable if and only if it is first countable. In particular, all our groups will be metrizable. It is easy to
check that a compact metrizable group admits a translation-invariant metric. We will implicitly choose such a metric on all
our groups, which will usually be denoted by $d$.\\
\indent We point out that most of our results are stated and proved for arbitrary second-countable compact groups, while
some results in Section 4 are only obtained for compact Lie groups.\\
\ \\
\indent \textbf{Acknowledgements:} Most of this work was done while the author was visiting the Westf\"alische
Wilhelms-Universit\"at M\"unster in the summer of 2013, and while the author was participating in the Thematic Program
on Abstract Harmonic Analysis, Banach and Operator Algebras, at the Fields Institute for Research in Mathematical Sciences
at the University of Toronto, in January-June 2014. He wishes to thank both Mathematics departments for their
hospitality and for providing a stimulating research environment. \\
\indent The author would like to express his gratitude Hannes Thiel for several helpful discussions concerning the existence of local
cross sections for free actions.
He is especially indebted to Ilan Hirshberg and Chris Phillips for giving him access to a draft of their paper \cite{HP} long before
it was ready for publication, and for helpful discussions and correspondence. Finally, he would also like to thank Wilhelm Winter
for suggesting this line of research.

\section{Notation and preliminaries}

\indent We adopt the convention that $\{0\}$ is not a \uca, this is, we require that $1\neq 0$ in a \uca.
Homomorphisms of \ca s are always assumed to be $\ast$-homomorphisms. For a \ca\ $A$, we denote by $\Aut(A)$
the automorphism group of $A$. If $A$ is moreover unital, then $\U(A)$ denotes the unitary group of $A$.\\
\indent For a locally compact group $G$, an action of $G$ on $A$ is always assumed to be a \emph{continuous}
group homomorphism from $G$ into $\Aut(A)$, unless otherwise stated. If $\alpha\colon G\to\Aut(A)$ is an
action of $G$ on $A$, then we will denote by $A^\alpha$ the fixed-point subalgebra of $A$ under $\alpha$. \\
\indent We take $\N=\{1,2,\ldots\}$. The $p$-adic integers will not appear in this paper, so we write $\Z_n$ for the cyclic group
$\Z/ n\Z$. If $A$ is a \ca, we denote by $[\cdot,\cdot]\colon A\times A\to A$ the additive commutator,
this is, $[a,b]=ab-ba$ for all $a,b\in A$.\\

\subsection{Equivariant $K$-theory} We recall the definition of equivariant $K$-theory for compact
group actions. A thorough development of the theory can be found in \cite{phillips equivariant k-theory book},
whose notation we will follow. We denote the suspension of a \ca\ $A$ by $SA$.

\begin{df}\label{df: eqKthy} Let $G$ be a compact group, let $A$ be a unital \ca\ and let $\alpha\colon G\to\Aut(A)$
be an action. Denote by $\mathcal{P}_G(A)$ the set of all $G$-invariant projections
in all of the algebras of the form $\B(V)\otimes A$, for all \fd\ representations $\lambda\colon G\to \U(V)$
(we take the diagonal action of $G$ on $\B(V)\otimes A$). Two $G$-invariant
projections $p$ and $q$ in $\mathcal{P}_G(A)$ are said to be
\emph{equivariantly Murray-von Neumann equivalent} if there exists a $G$-invariant partial isometry $s$ in $\B(V,W)\otimes A$
such that $s^*s=p$ and $ss^*=q$. Given a
\fd\ representation $\lambda\colon G\to \U(V)$ of $G$ and a $G$-invariant projection $p\in \B(V)\otimes A$,
and to emphasize role played by the representation $\lambda$,
we denote the element in $\mathcal{P}_G(A)$ it determines by $(p,V,\lambda)$. We let $S_G(A)$ denote the set
of equivalence classes in $\mathcal{P}_G(A)$ with addition given
by direct sum. \\
\indent We define the \emph{equivariant $K_0$-group} of $A$, denoted $K_0^G(A)$, to be the Grothendieck group of
$S_G(A)$, and define the \emph{equivariant $K_1$-group} of $A$, denoted $K_1^G(A)$, to be $K_0^G(SA)$,
where the action of $G$ on $SA$ is trivial in the suspension direction.\end{df}

\begin{rem} The equivariant $K$-theory of $A$ is a module over the representation ring $R(G)$ of
$G$, which can be identified with $K_0^G(\C)$, with the
operation given by tensor product. The induced operation $R(G)\times K_\ast^G(A)\to K_\ast^G(A)$ makes $K_\ast^G(A)$ into an $R(G)$-module. \end{rem}

The following result is Julg's Theorem.

\begin{thm}\label{Julg} (See Theorem 2.6.1 in \cite{phillips equivariant k-theory book})
Let $G$ be a compact group, let $A$ be a unital \ca\ and let $\alpha\colon G\to\Aut(A)$
be an action. Then there is a
natural isomorphism
$$K_\ast^G(A)\cong K_\ast(A\rtimes_\alpha G).$$ \end{thm}

When the group $G$ is abelian, we denote its gual group by $\widehat{G}$. In this case, the representation ring $R(G)$ can be naturally identified with $\Z[\widehat{G}]$.
With this identification in mind, the $R(G)$-module
structure on $K_\ast(A\rtimes_\alpha G)$ can be described as follows (see Proposition 2.7.10 in
\cite{phillips equivariant k-theory book}). Given $\tau$ in $\widehat{G}$ and $\eta$ in
$K_\ast(A\rtimes_\alpha G)$, we have
$$\tau \cdot \eta = K_\ast(\widehat{\alpha}_\tau)(\eta).$$

Recall the following definition from \cite{phillips equivariant k-theory book}.

\begin{df}\label{df: augmentation ideal} Let $G$ be a compact group. We let $I_G$ denote the augmentation
ideal of $R(G)$; that is, the kernel of the map
$R(G)\to \Z$ which sends the class $[V]$ of a finite dimensional unitary representation $V$ of $G$, to
its dimension $\dim(V)$. \end{df}

Given $n$ in $\N$, we denote by $I^n_G$ the product of $I_G$ with itself $n$ times.

\subsection{Central sequence algebras} Let $A$ be a \uca. Let $\ell^\I(\N,A)$ denote the set of all bounded
sequences $(a_n)_{n\in\N}$ in $A$, endowed with the supremum norm
$$\|(a_n)_{n\in\N}\|=\sup_{n\in\N}\|a_n\|$$
and pointwise operations. Then $\ell^\I(\N,A)$ is a \uca, the unit being the constant sequence $1_A$. Let
$$c_0(\N,A)=\left\{(a_n)_{n\in\N}\in\ell^\I(\N,A)\colon \lim_{n\to\I}\|a_n\|=0\right\}.$$
Then $c_0(\N,A)$ is an ideal in $\ell^\I(\N,A)$, and we denote the quotient
$$\ell^\I(\N,A)/c_0(\N,A)$$
by $A_\I$. Write $\kappa_A\colon \ell^\I(\N,A)\to A_\I$ for the quotient map. We identify $A$ with the
unital subalgebra of $\ell^\I(\N,A)$ consisting of the constant sequences, and with a unital subalgebra of
$A_\I$ by taking its image under $\kappa_A$. We write $A_\I\cap A'$ for the relative commutant of $A$
inside of $A_\I$. \\
\indent If $\alpha\colon G\to\Aut(A)$ is an action of $G$ on $A$, then there are actions of $G$ on $A_\I$
and on $A_\I\cap A'$, both denoted by $\alpha_\I$. Note that unless the group $G$ is discrete, these actions
will in general not be continuous. We set
$$\ell^\I_\alpha(\N,A)=\{a\in \ell^\I(\N,A)\colon g\mapsto (\alpha_\I)_g(a) \ \mbox{ is continuous}\},$$
and $A_{\I,\alpha}=\kappa_A(\ell^\I_\alpha(\N,A))$. By construction, the restriction of
$\alpha_\I$ to $A_{\I,\alpha}$ is continuous.\\
\ \\
\indent We note that if $\varphi\colon A\to B$ is a unital \hm\ of \uca s, then $\varphi$ induces unital
\hm s $\ell^\I(\N,\varphi)\colon \ell^\I(\N,A)\to \ell^\I(\N,B)$ and $\varphi_\I\colon A_\I\to B_\I$. The
assignments $A\mapsto \ell^\I(\N,A)$ and $A\mapsto A_\I$ are functorial for unital \ca s and unital \hm s.
%We also have the following proposition, which relates the crossed product functor with the sequence algebra
%functor.

%\begin{prop} \label{prop: cp and sequence algebra commute}
%Let $A$ be a \uca, let $G$ be a compact group and let $\alpha\colon G\to\Aut(A)$ be an action. Then there is a canonical
%embedding
%\[A_{\I,\alpha}\rtimes_{\alpha_\I} G\hookrightarrow (A\rtimes_\alpha G)_\I.\]
%\end{prop}
%\begin{proof}
%Note that if $B$ is a \ca, then $M(B_\I)$ contains $(M(B))_\I$ unitally. The canonical maps $A\to M(A\rtimes_\alpha %G)$
%and $G\to M(A\rtimes_\alpha G)$ induce canonical maps
%\begin{align*} A_{\I,\alpha}&\to (M(A\rtimes_\alpha G))_\I\hookrightarrow M((A\rtimes_\alpha G)_\I) \ \ \mbox{and} %\\
%G&\to (M(A\rtimes_\alpha G))_\I\hookrightarrow M((A\rtimes_\alpha G)_\I)\end{align*}
%which satisfy the covariance condition for $\alpha_\I$. It follows from the universal property of the crossed %product
%$A_{\I,\alpha}\rtimes_{\alpha_\I} G$ that there is a map as in the statement, which is injective because so is
%$A_{\I,\alpha}\to (M(A\rtimes_\alpha G))_\I$.\end{proof}

%In the proposition above, note that the canonical embedding will in general not be surjective unless $G$ is %finite.\\
%\ \\
\indent Functoriality of the assignment $A\mapsto A_\I\cap A'$ is more subtle, since not every
map between \ca s induces a
map between the corresponding central sequences. In Lemma \ref{surjective maps and central sequence algebra}
and Lemma \ref{tensor product and central sequence algebra}, we show two instances in which this is indeed
the case. These two cases will be needed in the following section.

\begin{lemma}\label{surjective maps and central sequence algebra} Let $\varphi\colon A\to B$ be a surjective
unital homomorphism of unital \ca s. Then the restriction of $\varphi_\I$ to $A_\I\cap A'$ induces a unital
homomorphism $\varphi_\I\colon A_\I\cap A'\to B_\I\cap B'$.\\
\indent Moreover, if $G$ is a locally compact group and $\alpha\colon G\to\Aut(A)$ and $\beta\colon G\to \Aut(B)$
are continuous actions of $G$ on $A$ and $B$ respectively, then $\varphi$ also induces a unital homomorphism
$\varphi_\I\colon A_{\I,\alpha}\cap A'\to B_{\I,\beta}\cap B'$, which is equivariant if $\varphi\colon A\to B$ is.
\end{lemma}
\begin{proof} We only need to check that $\varphi_\I(A_\I\cap A')\subseteq B_\I\cap B'$. Let
$a=(a_n)_{n\in\N}$ in $A_\I\cap A'$ and let $b\in B$. We have to show that $\varphi_\I(a)$ commutes with
$\kappa_B(b)$. Choose $c\in A$ such that $\varphi(c)=b$. Since $\kappa_B\circ\varphi=\varphi_\I\circ\kappa_A$,
we have
$$[\varphi_\I(a),b]=[\varphi_\I(a),\kappa_B(\varphi(c))]=[a,\kappa_A(c)]=0,$$
and the result follows.\\
\indent The proof of the second claim is analogous. \end{proof}

An important case of when a unital \hm\ between \uca s induces a unital \hm\ between the central sequence
algebras is that of the unital inclusion $A\hookrightarrow A\otimes B$ of a unital \ca\ as the first tensor
factor. This \hm\ is not covered by the previous lemma, so we shall prove it separately.

\begin{lemma}\label{tensor product and central sequence algebra} Let $A$ and $B$ be \uca s, let $A\otimes B$ be
any \ca\ completion of the algebraic tensor product of $A$ and $B$, and let
$\iota\colon A\to A\otimes B$ be given by $\iota(a)=a\otimes 1$ for all $a\in A$. Then $\iota_\I$ restricts
to a unital homomorphism
$$\iota_\I\colon A_\I\cap A'\to (A\otimes B)_\I\cap (A\otimes B)'.$$
\indent Moreover, if $G$ is a locally compact group and $\alpha\colon G\to\Aut(A)$ and $\beta\colon G\to \Aut(B)$
are continuous actions of $G$ on $A$ and $B$ respectively, and if the tensor product action
$$g\mapsto(\alpha\otimes\beta)_g=\alpha_g\otimes\beta_g$$
extends to $A\otimes B$, then $\iota$ induces a unital equivariant homomorphism
$$\iota_\I\colon A_{\I,\alpha}\cap A'\to (A\otimes B)_{\I,\alpha\otimes\beta}\cap (A\otimes B)'.$$
\end{lemma}
\begin{proof} Let $a=(a_n)_{n\in\N}$ in $A_\I\cap A'$ and let $x\in A\otimes B$. We may assume that $x$ is
a simple tensor, say $x=c\otimes b$ for some $c\in A$ and some $b\in B$. Then
$$[\iota_\I(a),x]=[(a_n\otimes 1)_{n\in\N},\kappa_{A\otimes B}(c\otimes b)]=0,$$
since $\lim\limits_{n\to\I}\|[a_n,c]\|= 0$.\\
\indent The proof of the second claim is straightforward. \end{proof}

\subsection{Completely positive order zero maps} We briefly recall some of the basics of completely positive
order zero maps. See \cite{winter zacharias} for more details and further results.\\
\ \\
\indent Let $A$ be a \ca, and let $a,b$ be elements in $A$. We say that $a$ and $b$ are \emph{orthogonal},
and write $a\perp b$, if $ab=ba=a^*b=ab^*=0$. If $a,b\in A$ are selfadjoint, then they are orthogonal \ifo
$ab=0$.

\begin{df} Let $A$ and $B$ be \ca s, and let $\varphi\colon A\to B$ be a completely positive map. We say
that $\varphi$ has \emph{order zero} if for every $a$ and $b$ in $A$, we have $\varphi(a)\perp \varphi(b)$
whenever $a\perp b$.\end{df}

\begin{rem} 
It is straightforward to check that $C^*$-algebra \hm s have order zero, and that the composition of two
order zero maps is again order zero.\end{rem}

The following is the main result in \cite{winter zacharias}.

\begin{thm}\label{cpc order zero and homs} (Theorem 2.3 and Corollary 3.1 in \cite{winter zacharias}) Let
$A$ and $B$ be \ca s. There is a bijection between completely positive contractive order zero maps $A\to B$
and $C^*$-algebra \hm s $C_0((0,1])\otimes A\to B$. A completely positive contractive order zero map
$\varphi\colon A\to B$ induces the \hm\ $\rho_\varphi\colon C_0((0,1])\otimes A\to B$ determined by
$\rho_\varphi(\id_{(0,1]}\otimes a)=\varphi(a)$ for all $a\in A$. Conversely, if $\rho\colon C_0((0,1])\otimes
A\to B$ is a \hm, then the induced completely positive contractive order zero map $\varphi_\rho\colon A\to B$
is the one given by $\varphi_\rho(a)=\rho(\id_{(0,1]}\otimes a)$ for all $a\in A$.\end{thm}

The following easy corollary will be used throughout without reference.

\begin{cor} \label{cor: cpcoz equivariant}
Let $A$ and $B$ be unital \ca s, let $G$ be a locally compact group, let $\alpha\colon G\to\Aut(A)$
 and $\beta\colon G\to\Aut(B)$ be continuous actions of $G$ on $A$ and $B$ respectively, and let $\varphi\colon
A\to B$ be a completely positive order zero map. Denote by $\rho_\varphi\colon C_0((0,1])\otimes A\to B$ the
induced homomorphism given by Theorem \ref{cpc order zero and homs}. Give $C_0((0,1])$ the
trivial action of $G$, and give $C_0((0,1])\otimes A$ the corresponding diagonal action. Then $\varphi$ is
equivariant if and only if $\rho_\varphi$ is equivariant.
\end{cor}
\begin{proof} We denote by $\widetilde{\alpha}\colon G\to\Aut(C_0((0,1])\otimes A)$ the diagonal action
described in the statement. Assume that $\rho$ is equivariant. Given $g$ in $G$ and $a$ in $A$, we have
\begin{align*}\varphi_\rho(\widetilde{\alpha}_g(\id_{(0,1]}\otimes a)))&=\varphi_\rho(\id_{(0,1]}\otimes \alpha_g(a)))\\
&= \rho(\alpha_g(a))\\
&= \beta_g(\rho(a))\\
&= \beta_g( \varphi_\rho(\id_{(0,1]}\otimes a))).
\end{align*}
Since $\id_{(0,1]}$ generates $C_0((0,1])$, we conclude that $\varphi_\rho$ is equivariant. \\
\indent Conversely, if $\varphi_\rho$ is equivariant, it is clear that its restriction to the invariant tensor
factor $A$ is also equivariant. This finishes the proof.
\end{proof}

\section{Rokhlin dimension for compact group actions}

We begin by recalling the definition of finite Rokhlin dimension for finite groups.

\begin{df}(See Definition 1.1 in \cite{HWZ}.)\label{df: Rdimfingp} Let $G$ be a finite group, let $A$ be a \uca, and let $\alpha\colon
 G\to\Aut(A)$ be an action of $G$ on $A$. Given a non-negative integer $d$, we say that $\alpha$ has \emph{Rokhlin
dimension $d$}, and denote this by $\dimRok(\alpha)=d$, if $d$ is the least integer with the following property:
for every $\ep>0$ and for every finite subset $F$ of $A$, there exist positive contractions $f_g^{(\ell)}$ for $g\in G$
and $\ell=0,\ldots,d$, satisfying the following conditions for every $\ell=0,\ldots,d$, for every $g,h\in G$, and for
every $a\in F$:
\be \item $\left\|\alpha_h\left(f_g^{(\ell)}\right)-f_{hg}^{(\ell)}\right\|<\ep$;
\item $\left\|f_g^{(\ell)}f_h^{(\ell)}\right\|<\ep$ whenever $g\neq h$;
\item $\left\|\sum\limits_{g\in G}\sum\limits_{\ell=0,\ldots,d} f_g^{(\ell)}-1\right\|<\ep$;
\item $\left\|\left[f_g^{(\ell)},a\right]\right\|<\ep$.\ee
If one can always choose the positive contractions $f_g^{(\ell)}$ above to moreover satisfy
$$\left\|\left[f_g^{(\ell)},f_h^{(k)}\right]\right\|<\ep$$
for every $g,h\in G$ and every $k,\ell=0,\ldots,d$, then we say that $\alpha$ has \emph{Rohlin dimension $d$ with
commuting towers}, and denote this by $\cdimRok(\alpha)=d$.
\end{df}

Given a compact group $G$, we denote by $\verb'Lt'\colon G\to\mbox{Homeo}(G)$ the action of left translation.
With a slight abuse of notation, we will also denote by $\verb'Lt'$ the induced action of $G$ on $C(G)$.\\
\indent Definition \ref{df: Rdimfingp} can be generalized to the case of second countable compact groups as follows.

\begin{df}\label{def finite Rdim}
Let $G$ be a second countable, Hausdorff compact group, let $A$ be a \uca,
and let $\alpha\colon G\to\Aut(A)$ be a continuous action. We say that $\alpha$ has \emph{\Rdim\ $d$}, if
$d$ is the least integer such that there exist equivariant completely positive contractive order zero
maps
$$\varphi_0,\ldots,\varphi_d\colon (C(G),\texttt{Lt})\to (A_{\I,\alpha}\cap A',\alpha_\I)$$
such that $\varphi_0(1)+\ldots+\varphi_d(1)=1$.\\
\indent We denote the \Rdim\ of $\alpha$ by $\dimRok(\alpha)$. If no integer $d$ as above exists, we say that
$\alpha$ has \emph{infinite Rokhlin dimension}, and denote it by $\dimRok(\alpha)=\I$.
If one can always choose the maps $\varphi_0,\ldots,\varphi_d$ to have commuting ranges, then we say that
$\alpha$ has \emph{Rokhlin dimension $d$ with commuting towers}, and write $\cdimRok(\alpha)=d$.\end{df}

\begin{rem}
It is an easy exercise to check that if $G$ is a finite group, then Definition \ref{def finite Rdim}
 agrees with Definition 1.1 in \cite{HWZ}. \end{rem}

It is clear that if $A$ is commutative, then the notions of \Rdim\ with and without commuting towers agree.
Nevertheless, Example \ref{eg: commuting matters} below shows that commuting towers cannot always be arranged,
even for $\Z_2$-actions on $\Ot$ with \Rdim\ 1. In fact, it seems that there really is a big difference between
these two notions, although we do not know how much they differ in general.

\begin{rem}\label{rem: Rp and Rdim}
It follows from Theorem 2.3 in \cite{winter zacharias} that a unital completely positive
contractive order zero map is necessarily a homomorphism. In particular, \Rdim\ zero is equivalent to the \Rp\
as in Definition 2.3 of \cite{hirshberg winter}. (We point out that the requirement that $\varphi$ be injective
in Definition 2.3 of \cite{hirshberg winter} is unnecessary: its kernel is a translation-invariant ideal of $C(G)$,
so it must be either $\{0\}$ or $C(G)$. Since $\varphi$ is assumed to be unital, it must be $\ker(\varphi)=\{0\}$.) \end{rem}

The following result is probably well-known to the experts. Since we have not been able to find a reference, we
present its proof here.

\begin{lemma} Let $A$ and $B$ be \ca s, and let $\varphi\colon A\to B$ be a completely positive contractive
order zero map. Then
$$\ker(\varphi)=\{a\in A\colon \varphi(a)=0\}$$
is a closed two-sided ideal in $A$.\end{lemma}
\begin{proof} That $\ker(\varphi)$ is closed follows easily by continuity of $\varphi$. Let us now show that it is a two-sided
ideal.\\
\indent Let $\pi\colon C_0((0,1])\otimes A\to B$ be the homomorphism determined by $\pi(\id_{(0,1]}\otimes a)
=\varphi(a)$ for all $a\in A$ (see Theorem \ref{cpc order zero and homs} above). Then $\ker(\pi)$ is an ideal
of $C_0((0,1])\otimes A$. Let $a\in \ker(\varphi)$ and let $x\in A$, and assume that $a$ is positive. Then
$ax$ belongs to $\ker(\varphi)$ if and only if $\id_{(0,1]}\otimes ax$ belongs to $\ker(\pi)$. Denote by
$t^{1/2}$ the map $(0,1]\to (0,1]$ given by $x\mapsto \sqrt{x}$. By functional calculus, $t^{1/2}\otimes a^{1/2}$
belongs to $\ker(\pi)$. It follows that
$$\id_{(0,1]}\otimes ax=\left(t^{1/2}\otimes a^{1/2}\right)\left(t^{1/2}\otimes a^{1/2}x\right)\in \ker(\pi)$$
since $\ker(\pi)$ is an ideal in $C_0((0,1])\otimes A$, and hence $\varphi(ax)=\pi(\id_{(0,1]}\otimes ax)=0$.
A similar argument shows that $\varphi(xa)=0$ as well, proving that $\ker(\varphi)$ is a two-sided ideal in $A$.\end{proof}

\begin{cor} Adopt the notation of Definition \ref{def finite Rdim} above. Then the order zero maps
$\varphi_0,\ldots,\varphi_d$ are either zero or injective.\end{cor}
\begin{proof} For $j=0,\ldots,d$, the kernel $I_j$ of $\varphi_j$ is a translation invariant ideal in $C(G)$,
since $\varphi_j$ is equivariant. The result now follows from the fact that the only translation invariant
ideals of $C(G)$ are $\{0\}$ and $C(G)$.\end{proof}

In particular, if $\dimRok(\alpha)=d<\I$, then the maps $\varphi_0,\ldots,\varphi_d$ from Definition
\ref{def finite Rdim} are injective.\\
\ \\
\indent We start by presenting some permanence properties for actions of compact groups with finite Rokhlin
dimension. Not surprisingly, finite Rokhlin dimension is far more flexible than the \Rp, and it is
preserved by several constructions. Most notably, finite Rokhlin dimension for finite dimensional compact groups
(in particular, for Lie groups) is inherited by the restriction to any closed subgroup, except that the actual
dimension may increase.\\
\indent We begin with a technical lemma which characterizes finite Rokhlin dimension in terms of elements in the
\ca\ itself, rather than its central sequence algebra.

\begin{lemma}\label{lem: characterization of finite Rdim}
Let $G$ be a compact group, let $A$ be a separable \uca, and let $\alpha\colon G\to\Aut(A)$ be an
action of $G$ on $A$. Let $d$ be a non-zero integer.
\be\item We have $\dimRok (\alpha)\leq d$ if and only if for every $\ep>0$, for every finite subset $S$ of $C(G)$,
and for every compact subset $F$ of $A$,
there exist completely positive contractive maps $\psi_0,\ldots,\psi_d\colon C(G)\to A$
satisfying the following conditions:
\be\item $\|\psi_j(f)a-a\psi_j(f)\|<\ep$ for all $j=0,\ldots,d$, for all $f$ in $S$, and all $a$ in $F$.
\item $\|\psi_j(\texttt{Lt}_g(f))-\alpha_g(\psi_j(f))\|<\ep$ for all $j=0,\ldots,d$, for all $g$ in $G$, and
for all $f$ in $S$.
\item $\left\|\psi_j(f_1)\psi_j(f_2)\right\|<\ep$ whenever $f_1$ and $f_2$ in $S$ are orthogonal.
\item $\left\|\sum\limits_{j=0}^d\psi_j(1_{C(G)})-1_A\right\|<\ep$.\ee
\item We have $\cdimRok (\alpha)\leq d$ if and only if for every $\ep>0$, for every finite subset $S$ of $C(G)$,
and for every compact subset $F$ of $A$, there exist completely positive contractive maps
$\psi_0,\ldots,\psi_d\colon C(G)\to A$ satisfying the conditions listed above in addition to
$$\|\psi_j(f_1)\psi_k(f_2)-\psi_k(f_2)\psi_j(f_1)\|<\ep$$
for all $j,k=0,\ldots,d$ and all $f_1$ and $f_2$ in $S$.
\ee
\end{lemma}
\begin{proof} We prove (1) first. Assume that for every $\ep>0$, for every finite subset $S$ of $C(G)$,
and every finite subset $F$ of $A$,
there exist completely positive contractive maps $\varphi_0,\ldots,\varphi_d\colon C(G)\to A$
satisfying the conditions of the statement. Choose increasing sequences $(F_n)_{n\in\N}$ and $(S_n)_{n\in\N}$
of finite subsets of $A$ and $C(G)$, whose union is dense in $A$ and in $C(G)$, respectively. Let
$\psi_0^{(n)},\ldots,\psi_d^{(n)}\colon C(G)\to A$ be as in the statement for the choices $F_n$ and
$\frac{1}{n}$. For $j=0,\ldots,d$, denote by $\varphi_j\colon C(G)\to A_{\I}$ the linear map given by
$$\varphi_j\left(\kappa_A((a_n)_{n\in\N})\right)=\kappa_A\left(\left(\psi_j^{(n)}(a_n)\right)_{n\in\N}\right)$$
for all $(a_n)_{n\in\N}$ in $\ell^\I(\N,A)$. Then $\varphi_j$ is easily seen to be completely positive
contractive and order zero. It is also straightforward to check that its image is contained in
$A_{\I,\alpha}\cap A'$, and that it is equivariant. Finally, it is immediate that $\sum\limits_{j=0}^d
\varphi_j(1)=1$.\\
\indent Conversely, suppose that $\alpha$ has Rokhlin dimension at most $d$. Choose completely
positive contractive order zero maps $\varphi_0,\ldots,\varphi_d\colon C(G)\to A_{\I,\alpha}\cap A'$ as
in the definition if finite Rokhlin dimension. Fix $j$ in $\{0,\ldots,d\}$. By Choi-Effros, there exist
completely positive contractive maps $\psi_j=(\psi_j^{(n)})_{n\in\N}\colon C(G)\to \ell^\I(\N,A)$ for
$j=0,\ldots,d$ such that for all $f,f_1,f_2$ in $C(G)$ with $f_1$ orthogonal to $f_2$, for all $a$ in
$A$, and for all $g$ in $G$, we have
\begin{align*}
\left\|\psi^{(n)}_j(f)a-a\psi^{(n)}_j(f)\right\|&\to 0\\
\left\|\psi^{(n)}_j(\mbox{\texttt{Lt}}_g(f))-\alpha_g(\psi^{(n)}_j(f))\right\|&\to  0\\
\left\|\psi^{(n)}_j(f_1)\psi^{(n)}_j(f_2)\right\|&\to 0 \\
\left\|\sum\limits_{j=0}^d\psi^{(n)}_j(1_{C(G)})-1_A\right\| & \to 0
\end{align*}
Given $\ep>0$, given a finite subset $S$ of $C(G)$ and given a finite subset $F$ of $A$, choose a
positive integer $n$ such that the quantities above are all less than $\ep$ on the elements of $S$
and $F$, respectively, and set $\psi_j=\psi^{(n)}_j$ for $j=0,\ldots,d$. This finishes the proof of (1).\\
\indent The proof of (2) is analogous. In particular, in the ``only if'' implication, one has to use
that the completely positive contractive maps $\psi_j=(\psi_j^{(n)})_{n\in\N}\colon C(G)\to \ell^\I(\N,A)$
for $j=0,\ldots,d$ obtained from Choi-Effros, moreover satisfy
$$\lim_{n\to\I} \left\|\psi^{(n)}_j(f_1)\psi^{(n)}_k(f_2)-\psi^{(n)}_k(f_2)\psi^{(n)}_j(f_1)\right\|=0$$
for all $f_1$ and $f_2$ in $C(G)$, and all $j,k=0,\ldots,d$. We omit the details.
\end{proof}

Regarding finite Rokhlin dimension as a noncommutative analog of freeness of group actions on topological
spaces, we give the following interpretation of Theorem \ref{thm: permanence properties} below. Part (1) is the analog of the fact that a diagonal action on a product space is free if only of the factors is free; part (2) is the analog of the fact that the restriction of a free action to an invariant closed subset is also free; and part (3) is the analog
of the fact that an inverse limit of free actions is again free.

\begin{thm}\label{thm: permanence properties}
Let $A$ be a unital \ca\, let $G$ be a compact group, and let $\alpha\colon G\to\Aut(A)$ be a continuous
action of $G$ on $A$.
\be
\item Let $B$ be a unital \ca, and let $\beta\colon G\to\Aut(B)$ be a continuous action of $G$ on $B$.
Let $A\otimes B$ be any \ca\ completion of the algebraic tensor product of $A$ and $B$ for which the tensor
product action $g\mapsto (\alpha\otimes\beta)_g=\alpha_g\otimes\beta_g$ is defined. Then
$$\dimRok(\alpha\otimes\beta)\leq \min \left\{\dimRok(\alpha),\dimRok(\beta)\right\}$$
and
$$\cdimRok(\alpha\otimes\beta)\leq \min \left\{\cdimRok(\alpha),\cdimRok(\beta)\right\}.$$
\item Let $I$ be an $\alpha$-invariant ideal in $A$, and denote by $\overline{\alpha}
\colon G\to \Aut(A/I)$ the induced action on the quotient. Then
$$\dimRok(\overline{\alpha})\leq \dimRok(\alpha)$$
and
$$\cdimRok(\overline{\alpha})\leq \cdimRok(\alpha).$$\ee
Furthermore,
\be
\setcounter{enumi}{2}
\item Let $(A_n,\iota_n)_{n\in\N}$ be a direct system of unital \ca s with unital
connecting maps, and for each $n\in\N$, let $\alpha^{(n)}\colon G\to\Aut(A_n)$ be a
continuous action such that $\iota_n\circ\alpha^{(n)}_g=\alpha^{(n+1)}_g\circ\iota_n$ for all $n\in\N$ and all
$g\in G$. Suppose that $A=\varinjlim A_n$ and $\alpha=\varinjlim \alpha^{(n)}$. Then
$$\dimRok(\alpha)\leq \liminf\limits_{n\to\I} \dimRok(\alpha^{(n)})$$
and
\[\cdimRok(\alpha)\leq \liminf\limits_{n\to\I} \cdimRok(\alpha^{(n)}).\]\ee \end{thm}

\begin{proof} We only prove the results for the noncommuting tower version; the proofs for the commuting tower
version are analogous and are left to the reader.\\
\indent Part (1). The statement is immediate if both $\alpha$ and $\beta$ have infinite \Rdim, so assume that
$\dimRok(\alpha)=d<\I$. Then there are equivariant completely positive contractive order zero maps
$$\varphi_0,\ldots,\varphi_d\colon C(G)\to A_{\I,\alpha}\cap A',$$
such that $\varphi_0(1)+\ldots+\varphi_d(1)=1$. Denote by $\iota\colon A\to A\otimes B$ the canonical embedding
as the first tensor factor. By Lemma \ref{tensor product and central sequence algebra}, this inclusion induces
a unital \hm\ $\iota_\I\colon A_\I\cap A'\to (A\otimes B)_{\I,\alpha\otimes\beta}\cap (A\otimes B)'$, and $\iota_\I$ is moreover
equivariant with respect to $\alpha_\I$ and $(\alpha\otimes\beta)_\I$. For $j=0,\ldots d$, set
$$\psi_j=\iota_\I\circ \varphi_j\colon C(G)\to (A\otimes B)_{\I,\alpha\otimes\beta}\cap (A\otimes B)'.$$
Then $\psi_0,\ldots,\psi_d$ are equivariant completely positive contractive order zero maps, and
$\psi_0(1)+\ldots+\psi_d(1)=\varphi_0(1)+\ldots+\varphi_d(1)=1$. Hence $\dimRok(\alpha\otimes\beta)\leq d$ and
the result follows.\\
\ \\
\indent Part (2). The statement is immediate if $\alpha$ has infinite \Rdim, so suppose that there exist a
positive integer $d$ in $\N$ and equivariant completely positive contractive order zero maps
$$\varphi_0,\ldots,\varphi_d\colon C(G)\to A_{\I,\alpha}\cap A'$$
such that $\varphi_0(1)+\ldots+\varphi_d(1)=1$. Denote by $\pi\colon A\to A/I$ the quotient map. Lemma
\ref{surjective maps and central sequence algebra} implies that $\pi$ induces a unital \hm
$$\pi_\I\colon A_\I\cap A'\to (A/I)_{\I,\overline{\alpha}}\cap (A/I)'.$$
Moreover, this homomorphism is easily seen to be equivariant. For
$j=0,\ldots,d$, set $\psi_j=\pi_\I\circ\varphi_j\colon C(G) \to (A/I)_{\I,\overline{\alpha}}\cap (A/I)'$. Then $\psi_j$ is an
equivariant completely positive contractive order zero map for all $j=0,\ldots,d$, and $\psi_0(1)+\ldots+\psi_d(1)
=\varphi_0(1)+\ldots+\varphi_d(1)=1$. It follows that $\dimRok(\overline{\alpha})\leq d$, as desired.\\
\ \\
\indent Part (3). The statement is immediate if $\liminf\limits_{n\to\I} \dimRok(\alpha^{(n)})=\I$. We shall therefore
assume that there exists $d\in\N$ such that for all $m\in\N$, there is $n\geq m$ in $\N$ with
$\dimRok(\alpha^{(n)})\leq d$. By passing to a subsequence, we may also assume that $\dimRok(\alpha^{(n)})\leq d$
for all $n\in\N$. \\
\indent We use Lemma \ref{lem: characterization of finite Rdim}. Let $\ep>0$, let $S$ be a finite subset of $C(G)$,
and let $F$ be a finite subset of $A$. With $L=\mbox{card}(F)$, write $F=\{a_1,\ldots,a_L\}$, and find a positive
integer $n$ in $\N$ and elements $b_1,\ldots,b_L$ in $A_n$ such that $\|a_j-b_j\|<\frac{\ep}{2}$ for all
$j=1,\ldots,L$. Choose completely positive contractive maps $\psi_0,\ldots,\psi_d\colon C(G)\to A_n$ satisfying
conditions (a) through (d) in part (1) of Lemma \ref{lem: characterization of finite Rdim} for $\frac{\ep}{2}$ and
the finite set $F'=\{b_1,\ldots,b_L\}\subseteq A_n$. If $\iota_{n,\I}\colon A_n\to A$ denotes the canonical map,
then it is easy to check that the completely positive contractive maps
$$\iota_{n,\I}\circ \psi_0,\ldots,\iota_{n,\I}\circ \psi_d\colon C(G)\to A$$
satisfy conditions (a) through (d) in part (1) of Lemma \ref{lem: characterization of finite Rdim} for $\ep$ and
the finite set $F$. This shows the result in the case of non commuting towers. \\
\indent The proof in the commuting towers case is analogous, using also the extra condition in part (2) of Lemma
\ref{lem: characterization of finite Rdim}. We omit the details.
\end{proof}

We point out that finite Rokhlin dimension does not pass to invariant subalgebras, even if the original action
has the \Rp, the \ca\ is $\Ot$, the invariant subalgebra is isomorphic to $\Ot$ and the action restricted to
the invariant subalgebra is pointwise outer, as the next example shows.

\begin{eg}\label{eg: gauge action Ot}
Denote by $s_1$ and $s_2$ the canonical generators of the Cuntz algebra $\mathcal{O}_2$.
Let $\gamma \colon \T\to\Aut(\Ot)$ be the gauge action, this is, the action determined by $\gamma_\zt(s_j)=\zt s_j$
for all $\zt\in\T$ and for $j=1,2$. Choose an isomorphism $\varphi \colon \Ot\otimes\Ot\to\Ot$ and let
$\alpha \colon\T\to\Aut(\Ot)$ be given by
\begin{align*}  \alpha_\zt= \varphi\circ(\gamma_\zt\otimes\id_\Ot)\circ\varphi^{-1}\end{align*}
for all $\zt\in\T$. It is shown in Example 3.7 in \cite{gardella classification on Otwo} that $\alpha$ has
the Rokhlin property. On the other hand, it is immediate that the restriction of $\alpha$ to the invariant
subalgebra $\varphi(1\otimes \Ot)$ does not have finite Rokhlin dimension, since it is conjugate to the trivial action on $\Ot$. We claim that the
restriction of $\alpha$ to the invariant subalgebra $\varphi(\Ot\otimes 1)$,
does not have finite Rokhlin dimension with commuting towers. Said restriction is conjugate to the gauge action
$\gamma$ on $\mathcal{O}_2$. To show that $\cdimRok (\gamma)=\I$, it is enough to show that no power of the augmentation
ideal $I_\T$ annihilates $K_\ast^\T(\mathcal{O}_2)$. Recall that the crossed product of $\mathcal{O}_2$ by the gauge action is isomorphic
to $M_{2^\I}\otimes\K$. For $n$ in $\N$, and under the canonical identifications given by Julg's Theorem
(here reproduced as Theorem \ref{Julg}), we have
$$I^n_\T\cdot K_\ast^\T(\mathcal{O}_2)\cong \text{Im}\left(\left(\id_{K_0(M_{2^\I}\otimes\K)}-K_\ast(\widehat{\gamma})\right)^n\right).$$
It is a well known fact that $\widehat{\gamma}$ is the unilateral shift on $M_{2^\I}\otimes\K$, whose induced action
on $K_0$ is multiplication by 2. Thus $\id_{K_0(M_{2^\I}\otimes\K)}-K_0(\widehat{\gamma})$ is multiplication by -1, which
is an isomorphism of $\Z\left[\frac{1}{2}\right]\cong K_\ast^\T(\mathcal{O}_2)$. In particular, any of its powers
is also an isomorphism, and thus
$$I^n_\T\cdot K_0^\T(\mathcal{O}_2)=K_0^\T(\mathcal{O}_2)\neq \{0\}$$
for all $n$ in $\N$. This proves the claim.
\end{eg}

\subsection{Restrictions to closed subgroups}
We now turn to restrictions of actions in relation to Rokhlin dimension. The following result is the analog of the
fact that the restriction of a free action to a (closed) subgroup is again free.

\begin{thm}\label{restrictions} Let $A$ be a \uca, let $G$ be a finite dimensional compact group, let $H$
be a closed subgroup of $G$, and let $\alpha\colon G\to\Aut(A)$ be a continuous action. Then
$$\dimRok(\alpha|_H)\leq (\dim(G)-\dim(H)+1)(\dimRok(\alpha)+1)-1$$
and
$$\cdimRok(\alpha|_H)\leq (\dim(G)-\dim(H)+1)(\cdimRok(\alpha)+1)-1$$
\end{thm}
\begin{proof} Without loss of generality, we may assume that $\dimRok(\alpha)$ is finite.\\
\indent Being a closed subspace of a finite dimensional subspace, $H$ is finite dimensional.
Let $d=\dim(G/H)=\dim(G)-\dim(H)$. We will produce $d+1$ completely positive contractive $H$-equivariant
order zero maps $\varphi_0,\ldots,\varphi_d\colon C(H)\to C(G)$ with $\varphi_0(1)+\cdots+\varphi_d(1)=1$.
Once we have done this, and since these maps will obviously have commuting ranges, both claims will follow by
composing each of the maps $\varphi_0,\ldots,\varphi_d$ with the $\dimRok(\alpha)+1$ maps
as in the definition of finite Rokhlin dimension for $\alpha$. The result will then be $(d+1)(\dimRok(\alpha)+1)$
maps which will satisfy the definition of finite Rokhlin dimension for $\alpha|_{H}$.\\
\indent Denote by $\pi\colon G\to G/H$ the canonical surjection. By part (1) of Theorem 2 in \cite{karube},
the map $\pi\colon G\to G/H$ is a principal $H$-bundle. In particular, there exist local cross-sections
from the orbit space $G/H$ to $G$. For every $x\in G$, let $V_{\pi(x)}$ be a neighborhood
of $\pi(x)$ in $G/H$ where $\pi$ is trivial. Using compactness of $G/H$, let $\mathcal{U}$ be a finite subcover of $G/H$. Use
Proposition 1.5 in \cite{kirchberg winter} to refine $\mathcal{U}$ to a $d$-decomposable covering $\mathcal{V}$. In other words,
$\mathcal{V}$ can be written as the disjoint union of $d+1$ families $\mathcal{V}_0\cup \cdots\cup \mathcal{V}_d$ of open sets, in such a
way that for every $k=0,\ldots,d$, the elements of $\mathcal{V}_k$ are pairwise disjoint.\\
\indent For $k=0,\ldots,d$, let
$V_k$ denote the union of all the open sets in $\mathcal{V}_k$, and note that there is a cross-section defined on $V_k$.
Let $\{f_0,\ldots,f_d\}$ be a partition of unity of $G/H$ subordinate to the cover $\{V_0,\ldots,V_d\}$.
Upon replacing $V_k$ with the open set $f_k^{-1}((0,1])\subseteq V_k$, we may assume that $V_k=f_k^{-1}((0,1])$ for all $k=0,\ldots,d$.
For each $k=0,\ldots,d$, set $U_k=\pi^{-1}(V_k)\subseteq G$, and observe that there is an equivariant homeomorphism
$U_k\cong V_k\times H$, where the $H$-action on $V_k\times H$ is diagonal with the trivial action on $V_k$ and translation
on $H$.
Define a continuous function $\phi_k\colon V_k\times H\cong U_k\to (0,1]\times H$ by
$$\phi_k(x,h)=(f_k(x),h)$$
for all $(x,h)$ in $V_k\times H\cong U_k$.
Then $\phi_k$ is continuous because the cross-section is continuous. Moreover, $\phi_k$ is clearly
equivariant. \\
\indent Identify $C_0((0,1])\otimes C(H)$ with $C_0((0,1]\times H)$, and for $k=0,\ldots,d$ define
$$\psi_k\colon C_0((0,1])\otimes C(H) \to C(G)$$
by
$$\psi_k(f)(x)=\left\{
                 \begin{array}{ll}
                   (f\circ \phi_k)(x), & \hbox{if $x\in U_k$;} \\
                   0, & \hbox{else.}
                 \end{array}
               \right.$$
Then $\psi_k$ is a homomorphism, and it is equivariant since $\phi_k$ is. The map $\varphi_k\colon C(H)\to
C(G)$ given by $\varphi_k(f)=\psi_k(\id_{(0,1]}\otimes f)$ for $f\in C(H)$ is an equivariant completely positive
contractive order zero map. Finally, using that $(f_k)_{k=0}^d$ is a partition of unity of $G/H$ at the last step, we have
$$\sum\limits_{k=0}^d\varphi_k(1)=\sum\limits_{k=0}^d\psi_k(\id_{(0,1]}\otimes 1)=\sum\limits_{k=0}^df_k=1.$$
It follows that the maps $\varphi_0,\ldots,\varphi_d$ have the desired properties, and the proof is finished.\end{proof}

In some cases, restricting to a subgroup does not increase the \Rdim. In the following proposition, the group
is not assumed to be finite dimensional.

\begin{prop} Let $A$ be a \uca, let $G$ be a compact group, and let $\alpha\colon G\to\Aut(A)$ be a continuous
action. Let $H$ be a closed subgroup of $G$, and assume that at least one of the following holds:
\be\item the coset space $G/H$ is zero dimensional (this is the case whenever $H$ has finite index in $G$).
\item $G=\prod\limits_{i\in I} G_i$ or $G=\bigoplus\limits_{i\in I} G_i$, and $H=G_j$ for some $j\in I$.
\item $H$ is the connected component of $G$ containing its unit. \ee
Then
$$\dimRok(\alpha|_H)\leq \dimRok(\alpha) \ \ \mbox{ and } \ \ \cdimRok(\alpha|_H)\leq \cdimRok(\alpha).$$\end{prop}
\begin{proof} In all these cases, we will produce a unital $H$-equivariant homomorphism $C(H)\to C(G)$, where
the $H$ action on both $C(H)$ and $C(G)$ is given by left translation. This is easily seen to be equivalent
to the existence of a continuous map $\phi\colon G\to H$ such that $\phi(hg)=h\phi(g)$ for all $h\in H$ and
all $g\in G$. \\
\indent Assuming the existence of such a \hm\ $C(H)\to C(G)$, the result will follow by composing it with
the completely positive contractive order zero maps associated with $\alpha$, similarly to what was done in
parts (1) and (2) of Theorem \ref{thm: permanence properties}.\\
\indent (1). Assume that $G/H$ is zero-dimensional. By Theorem 8 in \cite{mostert}, there exists a continuous section
$\lambda\colon G/H\to G$. Denote by $\pi\colon G\to G/H$ the quotient map, and define $\phi\colon G\to H$ by
$$\phi(g)=g(\lambda(\pi(g))^{-1}$$
for all $g\in G$. We check that the range of $\phi$, which a priori is contained in $G$, really lands in $H$:
$$\pi(\phi(g))=\pi(g)\pi\left(\lambda(\pi(g))^{-1}\right)=\pi(g)\pi(g)^{-1}=1$$
for all $g\in G$, so $\phi(G)\subseteq H$. Continuity of $\phi$ follows from continuity of $\lambda$ and from
continuity of the group operations on $G$. Finally, if $h\in H$ and $g\in G$, then
$$\phi(hg)=hg\lambda(\pi(hg))^{-1}=hg\lambda(\pi(g))^{-1}=h\phi(g),$$
as desired. \\
\indent (2). Both cases follow from the fact that there is a group homomorphism $G\to G_j$ determined by
$(g_i)_{i\in I}\mapsto g_j$.\\
\indent (3). This follows from (1) and the fact that $G/G_0$ is totally disconnected.\end{proof}

Rokhlin dimension can indeed increase when passing to a subgroup, even if the original action has the \Rp.

\begin{eg} \label{eg: Rdim not Rp}
Let $\alpha\colon\T\to\Aut(C(\T))$ be given by $\alpha_\zt(f)(\omega)=f(\zt^{-1}\omega)$ for
$\zt,\omega\in\T$ and $f\in C(\T)$. Then $\alpha$ has \Rdim\ zero. Given $n\in\N$ with $n>1$, identify $\Z_n$ with the
subgroup of $\T$ consisting of the $n$-th roots of unity. Then
$$\cdimRok(\alpha|_{\Z_n})=\dimRok(\alpha|_{\Z_n})=1.$$
Indeed, $\dimRok(\alpha|_{\Z_n})\leq 1$ by Theorem \ref{restrictions}. If $\dimRok(\alpha|_{\Z_n})=0$, then
$\alpha|_{\Z_n}$ would have the \Rp, which in particular would imply the existence of a non-trivial projection
in $C(S^1)$, which is a contradiction.\end{eg}

Even for circle actions with the \Rp, there are less obvious $K$-theoretic obstructions for the restriction
of a circle action with the Rokhlin property to have the Rokhlin property, besides merely the lack of projections,
as the next example shows.

\begin{eg}\label{eg: restrictions} (See Example 4.19 in \cite{gardella classification on Otwo}) There is an example of a purely infinite
simple separable nuclear \uca, and an action of the circle on it with the \Rp, such that no restriction to a
finite subgroup of $\T$ has the \Rp. \\
\indent To construct it, let $\{p_n\}_{n\in\N}$ be an enumeration of the prime numbers, and for every $n\in\N$ set
$q_n=p_1\cdots p_n$. Fix a countable dense subset $X=\{x_1,x_2,x_3,\ldots\}$ of $\T$ with $x_1=1$.
For $n$ in $\N$, define a unital injective map $\iota_n\colon M_{q_n}(C(\T))\to M_{q_{n+1}}(C(\T))$ by
$$\iota_n(f)=\left(
               \begin{array}{cccc}
                 f & 0 & \cdots & 0 \\
                 0 & \text{\texttt{Lt}}_{x_2}(f) & \cdots & 0 \\
                 \vdots & \vdots& \ddots & \vdots \\
                 0 & 0 & \cdots & \text{\texttt{Lt}}_{x_{p_n}}(f) \\
               \end{array}
             \right)$$
for $f$ in $M_{q_n}(C(\T))$. The direct limit $A=\varinjlim (M_{q_n}(C(\T)),\iota_n)$ is a simple unital A$\T$-algebra.
For $n\in\N$, let $\alpha^{(n)}\colon \T\to\Aut(M_{q_n}(C(\T)))$ be the tensor product of the trivial action on
$M_{q_n}$ with the action of left translation on $C(\T)$. Then the sequence $\left(\alpha^{(n)}\right)_{n\in\N}$
induces a direct limit action $\alpha=\varinjlim\alpha^{(n)}$ of $\T$ on $A$.\\
\indent Now set $B=A\otimes\OI$ and define $\beta\colon\T\to\Aut(B)$ by $\beta=\alpha\otimes\id_{\OI}$. Then $B$
is a purely infinite, simple, separable, nuclear \uca\ satisfying the UCT. It is shown in Example 4.19 in
\cite{gardella classification on Otwo} that $\beta$ has the \Rp, and that for every $m\in\N$ with $m>1$, the restriction
$\beta|_m\colon \Z_m\to\Aut(B)$ does not have the \Rp.\\
\end{eg}

We recall a result from \cite{gardella classification on Otwo} that gives a sufficient (and many times also necessary)
condition for the restriction of an action with the \Rp\ to have the \Rp.

\begin{thm} (Theorem 7.18 of \cite{gardella classification on Otwo}) Let $A$ be a separable \uca, let $n\in\N$ and let
$\alpha\colon\T\to\Aut(A)$ be an action with the \Rp. Suppose that $A$ absorbs $M_{n^\I}$. Then $\alpha|_{\Z_n}\colon\Z_n\to
\Aut(A)$ has
the \Rp.\end{thm}

We finish the discussion about restrictions of actions with finite Rokhlin dimension with an example
that shows that there is in general no way to determine the Rokhlin dimension of an action only by
knowing the Rokhlin dimension of all of its restrictions, at least in the case of commuting towers.
Our example is somewhat surprising: we exhibit an example of a circle action on a \ca, such that all of its
restrictions to finite subgroups have the Rokhlin property, but the action itself has infinite Rokhlin
dimension with commuting towers. (We do not know whether this action has finite Rokhlin dimension with
non-commuting towers.) 

\begin{eg}\label{eg: restr Rp action inf Rdim} (See Example 5.12 in \cite{gardella classification on Otwo}.)
Let $A$ be the universal UHF-algebra, this is, $A=\varinjlim(M_{n!},\iota_n)$ where $\iota_n\colon M_{n!}\to M_{(n+1)!}$ is
given by $\iota_n(a)=\diag(a,\ldots,a)$ for all $a$ in $M_{n!}$. For every $n\in\N$, let $\alpha^{(n)}\colon \T\to\Aut(M_{n!})$ be given by
$$\alpha^{(n)}_\zeta=\Ad(\diag(1,\zt,\ldots,\zt^{n!-1}))$$
for all $\zeta\in\T$. Then $\iota_n\circ\alpha^{(n)}_\zeta=\alpha^{(n+1)}_\zeta\circ\iota_n$ for all $n\in\N$ and all $\zeta\in\T$,
and hence there is a direct limit action $\alpha=\varinjlim\alpha^{(n)}$ of $\T$ on $A$. It is shown in
Example 5.12 in \cite{gardella classification on Otwo} that for every positive integer $m>1$, the restriction
$\alpha|_{\Z_m}\colon \Z_m\to \Aut(A)$ has the Rokhlin property, and that the action $\alpha$ itself does
not have the Rokhlin property. It moreover follows from Corollary \ref{cor: no actions with finite Rdim on AF}
that $\alpha$ does not even have finite Rokhlin dimension with commuting towers.\end{eg}

Example \ref{eg: restr Rp action inf Rdim} should be contrasted with the following fact.

\begin{prop} Let a compact Lie group $G$ act on a locally compact Hausdorff space $X$. Then the action is free
if and only if its restriction to every finite cyclic subgroup of $G$ of \emph{prime order} is free.\end{prop} 
\begin{proof} The ``only if" implication is immediate. For the ``if" implication, let $g\in G\setminus \{1\}$ and
assume that there exists $x$ in $X$ with $gx=x$. The stabilizer subgroup 
$$S_x=\{h\in G\colon hx=x\}$$
of $x$ is therefore non-trivial. Being a closed subgroup of $G$, it is a Lie group by Cartan's theorem. It follows that $S_x$ has a finite cyclic group of prime order: this is immediate if $S_x$ is finite, while if $S_x$ is infinite, it must
contain a (maximal) torus. Now, the restriction of the action to any such subgroup is trivial, contradicting the 
assumption. It follows that the action of $G$ on $X$ is free.\end{proof} 

\section{Comparison with other notions of non-commutative freeness}

In this section, we compare the notion of finite Rokhlin dimension (with and without commuting towers) with
some of the other forms of freeness of group actions on \ca s that have been studied. The properties we discuss
here include freeness of actions on compact Hausdorff spaces in the commutative case, the Rokhlin property,
discrete $K$-theory, local discrete $K$-theory, total $K$-freenesss, and pointwise outerness. \\
\indent We begin by comparing finite Rokhlin dimension on unital commutative \ca s with freeness of the
induced action on the maximal ideal space. Notice that in the case of commutative \ca s, the distinction
between commuting and non-commuting towers is irrelevant.

\begin{lemma} \label{lem: Rokdim on CX}
Let $G$ be a compact group acting on a compact Hausdorff space $X$. Denote by
$\alpha\colon G\to\Aut(C(X))$ the induced action of $G$ on $X$ and let $n$ be a non-negative integer.
Then $\alpha$ has Rokhlin dimension at most $n$ if and only if there are an open cover $\{U_0,\ldots,U_n\}$ of
$(\beta\N \setminus \N)\times X$ consisting of $G$-invariant open sets, and continuous, proper, equivariant functions
$\phi_j\colon U_j\to G\times (0,1]$, where the action of $G$ on $G\times (0,1]$ is translation on $G$ and
trivial on $(0,1]$.\end{lemma}
\begin{proof} Note that
$$C(X)_{\I,\alpha}\cap C(X)'= C(X)_{\I,\alpha} = C((\beta\N\setminus\N)\times X),$$
and that the induced action on $(\beta\N\setminus\N)\times X$ is trivial on $\beta\N\setminus\N$ and the
$G$-action on $X$. The existence of a completely positive contractive order zero map $\varphi\colon C(G)
\to C((\beta\N\setminus\N)\times X)$ is easily seen to be equivalent to the existence of an open set $U$
in $(\beta\N\setminus\N)\times X$ and a continuous function $\phi\colon U\to G\times (0,1]$. With this in
mind, it is easy to see that $\varphi$ is equivariant if and only if $U$ is $G$-invariant and $\phi$ is
equivariant. The rest of the proof is straightforward, and is omitted. \end{proof}

\begin{thm}\label{free action on spaces} Let $G$ be a compact Lie group and let $X$ be a compact Hausdorff
space. Let $G$ act on $X$ and denote by $\alpha\colon G\to\Aut(C(X))$ the induced action of $G$ on $C(X)$.
\be \item If $\alpha$ has finite Rokhlin dimension, then the action of $G$ on $X$ is free.
\item If the action of $G$ on $X$ is free, then $\alpha$ has finite \Rdim. In fact, there are a non-negative
integer $d$ and equivariant completely positive contractive order zero maps
$$\varphi_0,\ldots,\varphi_d\colon C(G)\to C(X)$$
such that $\sum\limits_{j=0}^d\varphi_j(1)=1$. Moreover, if $\dim(X)<\I$, we have
$$\dimRok(\alpha)\leq \dim(X)-\dim(G).$$
\ee\end{thm}
We point out that the conclusion in part (2) above really is stronger than $\alpha$ having finite Rokhlin dimension, since
one can choose the maps to land in $C(X)$ rather than in its (central) sequence algebra
$C(X)_{\I,\alpha}=C(X)_{\I,\alpha}\cap C(X)'$.
\begin{proof} Part (1). Assume that there exist $g\in G$ and $x\in X$ with $g\cdot x=x$.
Choose an open cover $U_0,\ldots,U_n$ of $(\beta\N \setminus \N)\times X$ consisting of $G$-invariant
open sets, and continuous equivariant functions $\phi_j\colon U_j\to G\times (0,1]$ as in Lemma
\ref{lem: Rokdim on CX}. Fix $\omega\in\beta\N\setminus\N$ and choose $j\in \{0,\ldots,n\}$ such that
$(\omega,x)\in U_j$. Write $\phi_j\colon U_j\to G\times (0,1]$ as $\phi=\left(\phi^{(1)},\phi^{(2)}\right)$.
Note that $\phi^{(1)}\colon U_j\to G$ is equivariant, where the action of $G$ on itself is given by left
translation (and in particular, it is free). We have
\begin{align*}
\left(\phi^{(1)}_j(\omega,x),\phi^{(2)}_j(\omega,x)\right)=\phi_j(\omega,x)=\phi_j(\omega,g\cdot x)=\left(g\phi^{(1)}_j(\omega,x),\phi^{(2)}_j(\omega,x)\right),
\end{align*}
which implies that $\phi^{(1)}_j(\omega,x)=g\phi^{(1)}_j(\omega,x)$ and hence $g=1$. The action of $G$ on $X$ is
therefore free.\\
\indent Part (2). The proof is almost identical to that of Theorem \ref{restrictions}, using Theorem 1.1 in \cite{phillips survey}
in place of part (1) of Theorem 2 in \cite{karube}. (Since $G$ is a Lie group, we do not need $X$ to be finite dimensional for the
quotient map $X\to X/G$ to be a principal $G$-bundle.) When $X$ is not necessarily finite-dimensional, we simply take $d$ to be
the carinality of some open cover $\mathcal{U}$ consisting of open subsets of $X/G$ over which the fiber bundle $X\to X/G$ is trivial.
When $\dim(X)<\I$, we have $\dim(X/G)=\dim(X)-\dim(G)$, so we can again use Proposition 1.5 in \cite{kirchberg winter} to refine
$\mathcal{U}$ to a $(\dim(X)-\dim(G))$-decomposable open cover of $X/G$, and proceed as in the proof of Theorem \ref{restrictions}.
We omit the details. \end{proof}

\begin{rem} It follows from the dimension estimate in part (2) of the above theorem that whenever a compact
Lie group $G$ acts freely on a compact Hausdorff space $X$ of the same dimension as $G$, then the induced action
of $G$ on $C(X)$ has the Rokhlin property. In fact, in this case it follows that $X$ is equivariantly homeomorphic
to $G\times (X/G)$, where the $G$-action on $G$ is left translation and the action on $X/G$ is trivial. Indeed,
$\dim(X/G)=\dim(X)-\dim(G)$, so $X/G$ is zero-dimensional. If $\pi\colon X\to X/G$ denotes the canonical quotient map,
then by Theorem 8 in \cite{mostert}, there exists a continuous map $\lambda\colon X/G \to X$ such that
$\pi\circ\lambda=\id_{X/G}$. One easily checks that the map $X\to G\times (X/G)$ given by
$x\mapsto (\lambda(\pi(x)),\pi(x))$ for $x\in X$, is a homeomorphism. It is also readily verified that it is
equivariant, thus proving the claim.
\end{rem}

Theorem \ref{strongly free and finite Rdim} below leads to a useful criterion to determine when a given action of a compact group has
finite Rokhlin dimension with commuting towers, although it is less useful if one is interested in the actual value of
the Rokhlin dimension. For most applications, however, having the exact value is not as important as knowing
that it is finite. In particular, it will follow from said theorem that for a compact Lie group,
the $X$-Rokhlin property for a finite dimensional compact Hausdorff space $X$ (as defined in Definition 1.5
in \cite{HP}), is equivalent to finite Rokhlin dimension with commuting towers in our sense.\\
\indent We first need a lemma about universal \ca s generated by the images of completely positive contractive
order zero maps. We present the non-commuting tower version, as well as its commutative counterpart, for use in a later
result. In the case of a finite group action with finite Rokhlin dimension with commuting towers,
the result below was first obtained by Ilan Hirshberg, and its proof is contained in the proof of Lemma 1.6 in \cite{HP}.

\begin{lemma} \label{lma: univ calg order zero maps}
Let $G$ be a compact group and let $d$ be a non-negative integer.
\be\item There exist a unital \ca\ $C$ and an action $\gamma\colon G\to\Aut(C)$ of $G$ on $C$ with the
following universal property. Let $B$ be a \uca, let $\beta\colon G\to\Aut(B)$ be an action of $G$ on $B$, and
let $\varphi_0,\ldots,\varphi_d\colon A \to B$ be equivariant completely positive contractive order zero maps
such that $\varphi_0(1)+\cdots+\varphi_d(1)=1$. Then there exists a unital equivariant
homomorphism $\varphi\colon C\to B$.
\item There exists a compact metrizable free $G$-space $X$ with the following universal property. Let $B$ be a \uca, let
$\beta\colon G\to\Aut(B)$ be an action of $G$ on $B$, and let $\varphi_0,\ldots,\varphi_d\colon A \to B$ be
equivariant completely positive contractive order zero maps with commuting ranges
such that $\varphi_0(1)+\cdots+\varphi_d(1)=1$. Then there exists a unital equivariant
homomorphism $\varphi\colon C(X) \to B$.\\
\indent Moreover, the space $X$ in part (2) satisfies
$$\dim(X)\leq (d+1)(\dim(G)+1)-1.$$\ee\end{lemma}
\begin{proof} Part (1). Set
$$D=\bigast _{j=0}^d C_0((0,1]\times G),$$
and let $\delta\colon G\to \Aut(D)$ be the action obtained by letting $G$ act on each of the free factors diagonally,
with trivial action on $(0,1]$ and translation on $G$. Denote by $I$ the ideal in $D$ generated by
$$\left\{\left(\sum\limits_{j=0}^d \id_{(0,1]}\ast 1_{C(G)}\right)c-c\colon c\in D\right\}.$$
Then $I$ is $\delta$-invariant, and hence there is an induced action $\gamma$ of $G$ on the unital quotient $C=D/I$. It is
clear that the \ca\ $C$ and the action $\gamma$ are as desired.\\
\indent Part (2). Set
$$D=\bigotimes\limits_{j=0}^d C_0((0,1]\times G),$$
and let $\delta\colon G\to \Aut(D)$ be the action obtained by letting $G$ act on each of the tensor factors diagonally.
Then $D$ is a commutative \ca\ and the action on its maximal ideal space induced by $\delta$ is free.
Denote by $I$ the ideal in $D$ generated by
$$\left\{\left(\sum\limits_{j=0}^d \id_{(0,1]}\otimes1_{C(G)} \right)c-c\colon c\in D\right\}.$$
Then $I$ is $\delta$-invariant, and hence there is an induced action $\gamma$ of $G$ on the unital quotient $C=D/I$. Set
$X=\text{Max}(C)$, which is a compact metrizable space. The action on $X$ induced by $\gamma$
is free, being the restriction of a free action to an invariant closed subset. It is clear that $X$ is the desired free $G$-space. \\
\indent The dimension estimate for $X$ follows from the fact that it is a closed subset of
$\bigotimes\limits_{j=0}^d (0,1]\times G$. \end{proof}

\begin{thm}\label{strongly free and finite Rdim} Let $G$ be a compact Lie group, let $A$ be a unital \ca, and let
$\alpha\colon G\to\Aut(A)$ be a continuous action. Then $\alpha$ has finite \Rdim\ with commuting towers \ifo
there exist a finite dimensional compact free $G$-space $X$ and an equivariant unital embedding
$$\varphi\colon C(X)\to A_{\I,\alpha}\cap A'.$$
Moreover, we have the following relations between the dimension of $X$ and the Rokhlin dimension of $\alpha$:
\begin{align*} \dim(X)&\leq (\cdimRok(\alpha)+1)(\dim(G)+1)-1 \\
\cdimRok(\alpha)&\leq \dim(X)-\dim(G)
\end{align*}
\end{thm}
\begin{proof} We begin by showing the ``only if'' implication.
Let $d=\cdimRok(\alpha)$, and denote by $Y$ the compact metrizable free $G$-space obtained as in the conclusion of
part (2) in Lemma \ref{lma: univ calg order zero maps}. By universality of $Y$, there is a unital equivariant
homomorphism
$$C(Y)\to A_{\I,\alpha}\cap A'.$$
The kernel of this homomorphism is a $G$-invariant ideal of $C(Y)$ which has the form $C_0(U)$ for some $G$-invariant
open subset $U$ of $Y$. Set $X=Y\setminus U$ and denote by $\varphi\colon C(X)\to A_{\I,\alpha}\cap A'$ the induced
homomorphism. Then the $G$-action on $X$ is free and $\varphi$ is unital, equivariant and injective. Finally, we have
$$\dim(X)\leq \dim(Y) \leq (d+1)(\dim(G)+1)-1,$$
and since $G$ is a compact Lie group, it follows that $X$ is finite dimensional.\\
\indent We now show the ``if'' implication. Set $d=\dim(X)-\dim(G)+1$ and choose completely positive contractive
order zero maps
$$\varphi_0,\ldots,\varphi_d\colon C(G)\to C(X)$$
as in the conclusion of part (2) of Theorem \ref{free action on spaces}. It is immediate to show that the
completely positive contractive order zero maps
$$\varphi\circ \varphi_0,\ldots,\varphi\circ \varphi_d\colon C(G)\to A_{\I,\alpha}\cap A'$$
satisfy the conditions in the definition of finite Rokhlin dimension with commuting towers for $\alpha$. This
finishes the proof.
\end{proof}

In particular, it follows from Theorem \ref{strongly free and finite Rdim} that for a compact Lie group,
the $X$-Rokhlin property for a compact Hausdorff space $X$ (as defined in Definition 1.5
of \cite{HP}), is equivalent to finite Rokhlin dimension with commuting towers.

\begin{cor}\label{cor: discrete K-thy}
Let $G$ be a compact Lie group, let $A$ be a \uca, and let $\alpha\colon G\to\Aut(A)$ be an action
with $\cdimRok(\alpha)<\I$. Then $\alpha$ has discrete $K$-theory, this is, there is $n$ in
$\N$ such that $I_G^n\cdot K_\ast^G(A)=0$. \end{cor}
\begin{proof} Corollary 4.3 in \cite{HP} asserts that compact Lie group actions with the $X$-Rokhlin
property have discrete $K$-theory. The result follows from the fact that finite Rokhlin dimension with
commuting towers implies the $X$-Rokhlin property by Theorem \ref{strongly free and finite Rdim}.
\end{proof}

It is proved in \cite{gardella rokhlin property} that a compact Lie group $G$ acts on a unital \ca\ $A$ with the \Rp,
then already $I_G$ annihilates $K_\ast^G(A)$. It may therefore be tempting to conjecture that in the context of
Corollary \ref{cor: discrete K-thy} above, one has
$$\cdimRok(\alpha)+1=\min\left\{n\colon I_G^n\cdot K_\ast^G(A)=0\right\},$$
or at least that the right-hand side determines $\cdimRok(\alpha)$. This is unfortunately not in general the
case, even for finite group actions on Kirchberg algebras that satisfy the UCT, as the following example shows.

\begin{eg}\label{eg: min does not determin dim}
Let $B$ and $\beta\colon\T\to\Aut(B)$ be the \ca\ and the circle action from Example \ref{eg: restrictions}.
As mentioned there, $\beta$ has the \Rp\ and for all $m$ in $\N$, its restriction $\beta|_m$ to $\Z_m\subseteq \T$ does not have
the \Rp. Fix $m$ in $\N$. It follows from Theorem \ref{restrictions} that $\cdimRok(\beta|_m)=1$. Moreover, by
Lemma 4.15 in \cite{gardella classification on Otwo} and part (1) of Proposition 4.22 in \cite{gardella classification on Otwo},
the dual action $\widehat{\beta|_n}\colon \Z_m\to \Aut(B\rtimes\Z_m)$ is approximately inner. In particular $1-K_\ast(\widehat{\beta|_n}_1)=0$
and thus $I_{\Z_m}\cdot K_\ast^{\Z_m}(B)=0$. If $\min\left\{n\colon I_{\Z_m}^n\cdot K_\ast^{\Z_m}(B)=0\right\}$ determined the Rokhlin
dimension of $\beta|_m$, we should have $\cdimRok(\beta|_m)=0$, and this would be a contradiction.
\end{eg}

The phenomenon exhibited above can also be encountered among free actions on spaces, as the action of $\Z_2$ on $S^1$ by
rotation shows. Finally, we mention that it is nevertheless conceivable that one has
$$\min\left\{n\colon I_G^n\cdot K_\ast^G(A)=0\right\}\leq \cdimRok(\alpha)+1,$$
but we have not explored this direction any further.\\
\ \\
\indent Having discrete $K$-theory is special to actions with finite Rokhlin dimension \emph{with commuting towers},
as the following example, which was constructed by Izumi in a different context, shows. \\
\ \\
\indent Let $p$ be a projection in $\OI$ whose class in $K_0(\OI)\cong \Z$ is 0. We recall that the
\emph{standard Cuntz algebra} $\mathcal{O}_\I^{st}$ is defined to be the corner $p\OI p$. It follows from
Kirchberg-Phillips classification of Kirchberg algebras in the UCT class (see \cite{kirchberg_classif} and
\cite{phillips_classif}), that $\mathcal{O}_\I^{st}$ is
the unique unital Kirchberg algebra satisfying the UCT with $K$-theory given by
$$(K_0(\mathcal{O}_\I^{st}),\left[1_{\mathcal{O}_\I^{st}}\right],K_1(\mathcal{O}_\I^{st}))\cong (\Z,0,\{0\}).$$
\indent Since the unit of $\mathcal{O}_\I^{st}$ is trivial in $K_0$, there is a unital homomorphism $\Ot\to \mathcal{O}_\I^{st}$. Hence there is an approximately central embedding of $\Ot$ into $\bigotimes\limits_{n\in\N}
\mathcal{O}_\I^{st}$, so it follows from Theorem 3.8 in \cite{KP} that $\bigotimes\limits_{n\in\N}\mathcal{O}_\I^{st}
\cong \Ot$.

\begin{eg}\label{eg: commuting matters}
(See the example on page 262 of \cite{izumi finite group actions I}.) Let $p$ be a projection in
$\mathcal{O}_\I$ whose class in $K_0(\OI)$ is 0, and set $u=2p-1$. Then $u$ is a unitary of $\OI$
which leaves the corner $p\OI p\cong \mathcal{O}_\I^{st}$ invariant. Since
$\bigotimes\limits_{n\in\N} \mathcal{O}_\I^{st}$ is isomorphic to $\Ot$, if we let $\alpha$ be the infinite tensor product
automorphism $\alpha=\bigotimes\limits_{n\in\N}\Ad(u)$ of $\Ot$, then $\alpha$ determines a $\Z_2$ action on $\Ot$.\\
\indent It is shown in Remark 2.5 in \cite{BEMSW} that $\alpha$ has Rokhlin dimension 1 with non-commuting
towers. We claim that $\text{dim}_{\text{Rok}}^c(\alpha)=\I$, this is, that $\alpha$ has infinite Rokhlin dimension
with commuting towers. We show that $\alpha$ does not have discrete $K$-theory, and the result will then
follow from Corollary \ref{cor: discrete K-thy}.\\
\indent It is shown in \cite{izumi finite group actions I} that $\Ot\rtimes_\alpha\Z_2$ is isomorphic to a
direct limit of $B_n=\mathcal{O}_\I^{st}\oplus \mathcal{O}_\I^{st}$ with connecting maps that on $K_0$ are
stationary and given by the matrix
$$\left(
               \begin{array}{cc}
                 1 & -1  \\
                 -1 & 1\\
               \end{array}
             \right).$$
Moreover, the dual action $\widehat{\alpha}\colon \widehat{\Z_2}\cong \Z_2\to\Aut(\Ot\rtimes_\alpha\Z_2)$ is
the direct limit of the actions $\gamma_n\colon \Z_2\to \Aut(B_n)$ given by $\gamma_n(a,b)=(b,a)$ for all
$(a,b)\in B_n=\mathcal{O}_\I^{st}\oplus \mathcal{O}_\I^{st}$. It follows that $\widehat{\alpha}$ is multiplication
by $-1$ on $K_0(\Ot\rtimes_\alpha\Z_2)$ (it is given by exchanging the columns in the above matrix). It is
shown in Lemma 4.7 in \cite{izumi finite group actions I} that
$K_0(\Ot\rtimes_\alpha\Z_2)\cong \Z\left[\frac{1}{2}\right]$. Given $n$ in $\N$, we have
$$I_{\Z_2}^n\cdot K_0^{\Z_2}(\Ot)\cong \text{Im}\left(\id_{K_0(\Ot\rtimes_\alpha\Z_2)}-K_0(\widehat{\alpha})\right)^n.$$
Now, $\id_{K_0(\Ot\rtimes_\alpha\Z_2)}-K_0(\widehat{\alpha})$ is multiplication by 2 on $\Z\left[\frac{1}{2}\right]$,
so
$$\left(\id_{K_0(\Ot\rtimes_\alpha\Z_2)}-K_0(\widehat{\alpha})\right)^n$$
is an isomorphism for all $n$ in $\N$, and in particular $\alpha$ does not have discrete $K$-theory.
This proves the claim.
\end{eg}

It follows from the example above that the notions of Rokhlin dimension with and without commuting towers do not in general agree.
Even more, having \emph{finite} Rokhlin dimension without commuting towers is really weaker than having \emph{finite}
Rokhlin dimension with commuting towers. Such phenomenon can happen even if the $K$-theoretic obstructions found in
\cite{HP} vanish. This answers a question that was implicitly left open in \cite{HWZ}, at least
in the finite (and compact) group case. We do not know whether there are similar examples for automorphisms.\\
\ \\
\indent Finite Rokhlin dimension with commuting towers is not in general equivalent to having discrete $K$-theory, since
the trivial action on $\Ot$ clearly does not have finite Rokhlin dimension but has discrete $K$-theory for trivial reasons.
The following example, which was originally constructed by Phillips with a different purpose, shows that absence of $K$-theory is not
the only thing that can go wrong.

\begin{eg}\label{eg: discr K-thy not Rdim} We recall the construction in Example 9.3.9 in \cite{phillips equivariant k-theory book}
of an AF-action of $\Z_4$ on the CAR algebra $M_{2^\I}$ whose restriction to $\Z_2$ is not $K$-free, and show that it has other
interesting properties.\\
\indent For $n$ in $\N$, let $A_n=M_{2^n}(\C\oplus \C)$ and set $u_n=\bigotimes\limits_{k=1}^n \diag(1,-1)$, which is a unitary in $M_{2^n}$ (not
in $A_n$). Define connecting maps $\iota_n\colon A_n\to A_{n+1}$ by
$$\iota_n(a,b)=\left( \left(
               \begin{array}{cc}
                 a & 0 \\
                 0 & b \\
               \end{array}
             \right),\left(
               \begin{array}{cc}
                 b & 0 \\
                 0 & u_nau_n^* \\
               \end{array}
             \right)
\right) $$
for $(a,b)\in A_n$. Define an automorphism $\alpha^{(n)}$ of $A_n$ by $\alpha^{(n)}(a,b)=(u_nbu_n^*,a)$
for $(a,b)\in A_n$. Since $u_n$ has order two, it is easy to see that $\alpha^{(n)}$ has order four, so it defines an action
of $\Z_4$ on $A_n$. It is also readily checked that there is a direct limit action $\alpha=\varinjlim \alpha^{(n)}$ of $\Z_4$ on
$A=\varinjlim A_n$. Finally, the direct limit algebra $A$ is easily seen to be isomorphic to the CAR algebra $M_{2^\I}$ by classification.\\
\indent As proved in \cite{phillips equivariant k-theory book}, with $p=(1,0)\in \C\oplus \C\subseteq A$, it is easy to show that
$\alpha^2$ is the action of conjugation by the unitary $2p-1$, so $\alpha|_{\Z_2}$ is in fact inner. In particular, $\alpha$ is not
pointwise outer so it does not have finite Rokhlin dimension, with or without commuting towers, by Theorem \ref{pointwise outer} below.\\
\indent The crossed product $A\rtimes_\alpha\Z_4$ is the direct limit of the inductive system
$$A_1\otimes C^*(\Z_4)\to A_2\otimes C^*(\Z_4)\to \cdots \to A\rtimes_\alpha \Z_4.$$
The computation of the connecting maps is routine, and yields an isomorphism $A\rtimes_\alpha\Z_4\cong M_{2^\I}$, which is best seen
using Bratteli diagrams. (Alternatively, one can
compute the equivariant $K$-theory of $A$, as is done in \cite{phillips equivariant k-theory book}.) To show that $\alpha$ has discrete
$K$-theory, it suffices to observe that the dual action acts via approximately inner automorphisms, since every automorphism of a
UHF-algebra is approximately inner. In particular,
$$I_{\Z_4}\cdot K_\ast^{\Z_4}(A)\cong \mbox{Im}\left(1-K_\ast(\widehat{\alpha}_1)\right)=0,$$
as desired.
\end{eg}

We turn to the comparison with locally discrete $K$-theory and total $K$-freeness.

\begin{df} (See Definitions 4.1.1, 4.2.1 and 4.2.4 of \cite{phillips equivariant k-theory book}.)
Let $G$ be a compact group, let $A$ be a unital \ca\ and let $\alpha\colon G\to\Aut(A)$ be a continuous action.
\be\item We say that $\alpha$ has \emph{locally discrete $K$-theory} if for every prime ideal $P$ of $R(G)$ not
containing the augmentation ideal $I_G$, the localization $K_\ast^G(A)_P$ is zero.
\item We say that $\alpha$ is \emph{$K$-free} if for every invariant ideal $I$ of $A$, the induced action
$\alpha|_I\colon G\to\Aut(I)$ has locally discrete $K$-theory.
\item We say that $\alpha$ is \emph{totally $K$-free} if for every closed subgroup $H$ of $G$, the restriction
$\alpha|_H$ is $K$-free.\ee
\end{df}

\begin{cor}
Let $G$ be a compact Lie\ group, let $A$ be a \uca\ and let $\alpha \colon G \to \Aut (A)$
be an action with finite Rokhlin dimension with commuting towers. Then $\alpha$ has locally discrete $K$-theory.
\end{cor}
\begin{proof}
The action $\alpha$ has discrete $K$-theory by Corollary \ref{cor: discrete K-thy}. It then follows from
the equivalence between (1) and (2) in Proposition 4.1.3 of \cite{phillips equivariant k-theory book} that
$\alpha$ has locally discrete $K$-theory.\end{proof}

\begin{cor}\label{cor: totKfree}

Let $G$ be a compact Lie group, let $A$ be a \uca\ and let $\alpha \colon G \to \Aut (A)$
be an action with finite Rokhlin dimension with commuting towers. Then $\alpha$ is totally $K$-free.
\end{cor}
\begin{proof}
Let $H$ be a closed subgroup of $G$ and let $I$ be an $H$-invariant ideal of $A$. Since the restriction
of $\alpha$ to $H$ has finite Rokhlin dimension with commuting towers by Theorem \ref{restrictions}, we
may assume that $H=G$, so that $I$ is $G$-invariant. We have to show that the induced action of $G$ on
$I$ has locally discrete $K$-theory.\\
\indent Since the induced action of $G$ on $A/I$ has finite Rokhlin dimension with commuting towers by
part (2) of Theorem \ref{thm: permanence properties}, it follows from the corollary above that it has
locally discrete $K$-theory. In particular, the extension
$$0\to I\to A\to A/I\to 0$$
is $G$-equivariant, and the actions on $A$ and $A/I$ have locally discrete $K$-theory. The result now
follows from Lemma 4.1.4 of \cite{phillips equivariant k-theory book}.\end{proof}

Even total $K$-freeness is not equivalent to finite Rokhlin dimension.

\begin{eg}\label{eg: tot K-free not Rdim} Let $\alpha$ be the trivial action of $\Z_2$ on $\Ot$. Then $\alpha$ is
 readily seen to be totally $K$-free, but it clearly does not have finite Rokhlin dimension, with or without
commuting towers, by Theorem \ref{pointwise outer} below.
\end{eg}

Recall that an action of a locally compact group $G$ on a \ca\ $A$ is said to be \emph{pointwise outer} (and
sometimes just \emph{outer}), if for every $g\in G\setminus \{1\}$, the automorphism $\alpha_g$ of $A$ is not inner.

\begin{thm}\label{pointwise outer} Let $A$ be a \uca, let $G$ be a compact Lie group, and let $\alpha\colon G
\to\Aut(A)$ be a continuous action. If $\dimRok(\alpha)<\I$, then $\alpha$ is pointwise outer.\end{thm}
We point out that we do not assume that $\alpha$ has finite Rokhlin dimension with commuting towers, unlike in most
other results in this section.
\begin{proof} Let $g\in G\setminus\{1\}$ and assume that $\alpha_{g_0}$ is inner, say $\alpha_{g_0}=\Ad(u)$ for some $u\in\U(A)$.
Let $C$, let $\gamma\colon G\to \Aut(C)$ and let
$$\varphi\colon C\to A_{\I,\alpha}\cap A'$$
be the unital \ca, the action of $G$ on $C$, and the unital equivariant homomorphism obtained as in the conclusion
of part (1) of Lemma \ref{lma: univ calg order zero maps}. \\
\ \\
\indent We claim that there exists a positive element $a$ in $C$ with the following properties:
\bi\item The elements $a$ and $\gamma_g(a)$ are orthogonal.
\item $\|\varphi(a)\|=\|\varphi(\gamma_g(a))\|=\|\varphi(a)-\varphi(\gamma_g(a))\|=1$.\ei
Set $d=\dimRok(\alpha)$. Recall from the proof of Lemma \ref{lma: univ calg order zero maps} that $C$ is the quotient of the \ca\
$$D=\bigast _{j=0}^d C_0((0,1]\times G)$$
by the ideal $I$ generated by
$$\left\{\left(\sum\limits_{j=0}^d\id_{(0,1]}\ast 1_{C(G)}\right)c-c\colon c\in D\right\}.$$
Denote by $\pi\colon D\to C$ the quotient map.
Choose a positive function $f$ in $C(G)$ such that the supports of $\texttt{Lt}_g(f)$ and $f$ are disjoint. Set
$b=\id_{(0,1]}\otimes f \in C_0((0,1])\otimes C(G)$, and regard it as an element in $D$ via the embedding of
$C_0((0,1])\otimes C(G)$ as the first free factor. We claim that $\pi(b)\neq 0$. Indeed, if $\pi(b)=0$, then
$\pi(\delta_h(b))=0$ for all $h$ in $G$. Since the action of translation of $G$ on itself is transitive, we conclude
that the first free factor of $D$ is contained in the kernel of $\pi$. Now, this contradicts the fact that $d=\dimRok(\alpha)$,
since it shows that the definition of finite Rokhlin dimension for $\alpha$ is satisfied with $d-1$ order zero maps. This shows
that $\pi(b)\neq 0$. \\
\indent Upon renormalizing $b$, we may assume that $a=\pi(b)$ is positive and has norm 1. It is clear that $a$ and
$\gamma_g(a)=\pi(\delta_g(b))$ are orthogonal, and that $\gamma_g(a)$ is positive and has norm 1. Finally, it follows
from orthogonality of $a$ and $\gamma_g(a)$ that $\|\varphi(a)-\varphi(\gamma_g(a))\|=1$. This proves the claim.\\
\ \\
\indent Let $\ep=\frac{1}{3}$. Using Choi-Effros lifting theorem, find a completely positive contractive map $\psi\colon C\to A$
satisfying the following conditions:
\be\item $\|[\psi(a),u]\|<\ep$;
\item $\|\psi(\gamma_g(a))-\alpha_g(\psi(a))\|<\ep$;
\item $|\|\psi(a)-\psi(\gamma_g(a))\|-1|<\ep$.\ee
We have
\begin{align*} \frac{2}{3}&=1-\ep<\|\psi(a)-\psi(\gamma_g(a))\| \leq \ep+\|\psi(a)-\alpha_g(\psi(a))\|\\
&= \ep+\|\psi(a)-u\psi(a)u^*\| \leq 2\ep=\frac{2}{3},\end{align*}
which is a contradiction. This contradiction implies that $\alpha_g$ is not inner, thus
showing that $\alpha$ is pointwise outer. \end{proof}

In the case of commuting towers, the converse to the theorem above fails quite drastically, and there are many examples of compact
group actions that are pointwise outer and have infinite Rokhlin dimension with commuting towers. See Example
\ref{eg: gauge action Ot}, where it is shown that the gauge action on $\Ot$ has infinite Rokhlin dimension with
commuting towers, and see Example \ref{eg: commuting matters} for an example where the acting group is $\Z_2$.
The second one has finite Rokhlin dimension with non-commuting towers (in fact, Rokhlin dimension 1). We do not
know whether the Rokhlin dimension of the gauge action on $\Ot$ is finite. It is known, however, that all of its
restrictions to finite subgroups of $\T$ have Rokhlin dimension with non commuting towers equal to 1.  \\
\indent On the other hand, we do not know exactly how badly the converse to the theorem above fails in the case of
non-commuting towers, although we know it does not hold in full generality.

\begin{eg}\label{eg: pointwise outer not Rdim} The action $\alpha$ of $\Z_2$ on $S^1$ given
by conjugation has two fixed points, so it is not free, and hence $\dimRok(\alpha)=\I$. On the other hand, $\alpha$ is
certainly pointwise outer since it is not trivial.\end{eg}

\subsection{A rigidity result}

Using the fact that compact group actions with finite Rokhlin dimension with commuting towers are totally
$K$-free by Corollary \ref{cor: totKfree}, we are able to obtain a certain $K$-theoretical obstruction for
a \uca\ to admit such an action. As a consequence of this $K$-theoretical obstruction, we confirm a conjecture
of Phillips.

\begin{thm}\label{K-thy restrictions for K-freeness} Let $G$ be a compact Lie group, let
$A$ be a \ca, and let $\alpha\colon G\to\Aut(A)$ be a continuous action. Assume that one and
only one of either $K_0(A)$ or $K_1(A)$ vanishes. If $\alpha$ is totally $K$-free, then
$G$ is finite.\end{thm}
\begin{proof} We will show the result assuming that $K_1(A)=0$; the corresponding proof for
the case $K_0(A)=0$ is analogous. \\
\indent By restricting to the connected component of the identity in $G$, and recalling that
$K$-freeness passes to subgroups, we can assume that $G$ is connected. Assume that $G$ is not
the trivial group. By further restricting to any copy of the circle inside a maximal torus,
we may assume that $G=\T$. Having discrete $K$-theory, there exists $n\in\N$ such that
$I_\T^n\cdot  K_\ast^\T(A)=0$. Equivalently,
$$\ker((\id_{K_\ast(A\rtimes_\alpha\T)}-\widehat{\alpha}_\ast)^n)=K_\ast(A\rtimes_\alpha\T).$$
Using that $K_1(A)=0$, it follows from
the Pimsner-Voiculescu exact sequence associated to $\alpha$ (see Subsection 10.6 in \cite{blackadar K-theory}),
\beqa\xymatrix{K_0(A\rtimes_\alpha\T)\ar[rr]^-{1-K_0(\widehat{\alpha})}&&K_0(A\rtimes_\alpha\T)\ar[rr]&&K_0(A)\ar[d]\\
K_1(A)\ar[u]&&K_1(A\rtimes_\alpha\T)\ar[ll]&&K_1(A\rtimes_\alpha\T)\ar[ll]^-{1-K_1(\widehat{\alpha})},}\eeqa
that the map $\id_{K_0(A\rtimes_\alpha\T)}-\widehat{\alpha}_0$ is injective. This implies that
$K_0(A\rtimes_\alpha\T)=0$ and the remaining potentially non-zero terms in the Pimsner-Voiculescu
exact sequence yield the short exact sequence
$$0\to K_0(A)\to K_1(A\rtimes_\alpha\T)\to K_1(A\rtimes_\alpha\T)\to 0,$$
where the last map is $\id_{K_1(A\rtimes_\alpha\T)}-\widehat{\alpha}_1$. Being surjective, every
power of it is surjective, and hence the identity
$$\ker((\id_{K_1(A\rtimes_\alpha\T)}-\widehat{\alpha}_1)^n)=K_1(A\rtimes_\alpha\T)$$
forces $K_1(A\rtimes_\alpha\T)=0$. In this case,
it must be $K_0(A)=0$ as well, which contradicts the fact that $K_0(A)$ is not zero.\end{proof}

Recall that an AF-action is an action on an AF-algebra obtained as a direct limit of actions on finite
dimensional \ca s. It was shown in \cite{blackadar symm of CAR alg} that not every action on an AF-algebra is
an AF-action, even when the group is $\Z_2$.

\begin{rem} \label{rem: conjecture}
Conjecture 9.4.9 in \cite{phillips equivariant k-theory book} says that there does not
 exist a totally $K$-free AF-action of a non-trivial connected compact Lie group on an AF-algebra.
Theorem \ref{K-thy restrictions for K-freeness} above confirms this conjecture of Phillips, for a
much larger class of \ca s, and without assuming that the action is specified by the way it is
constructed. \end{rem}

\begin{cor}\label{cor: no actions with finite Rdim on AF} No non-finite compact Lie group admits an action with finite Rokhlin dimension with commuting towers on a
\ca\ with exactly one vanishing $K$-group. In particular, there are no such actions on AF-algebras, AI-algebras, the Cuntz algebras $\mathcal{O}_n$ for $n\in\{3,\ldots,\I\}$, or the
Jiang-Su algebra $\mathcal{Z}$. \end{cor}

We make some comments about what happens for finite groups. Many AF-algebras (although not all of them) as well as all Cuntz
algebras $\mathcal{O}_n$ with $n< \I$, admit finite group actions with finite Rokhlin dimension.
In fact, they even admit actions of some finite groups with the \Rp\ (although there are severe
restrictions on the cardinality of the group in each case). On the other hand, Theorem 4.7
in \cite{HP} asserts that $\OI$ and $\mathcal{Z}$ do not admit \emph{any} finite group
action with finite Rokhlin dimension.\\
\ \\
\indent We specialize to AF-algebras now, since a little more can be said in this case. Recall
that an action $\alpha$ of
a locally compact group $G$ on a \uca\ is said to be \emph{inner} if there exists a continuous
group homomorphism $u\colon G\to\U(A)$ such that $\alpha_g=\Ad(u(g))$ for all $g\in G$.

\begin{df}\label{df: loc repr}
An AF-action $\alpha$ of a locally compact group $G$ on a \uca\ $A$ is said to be
\emph{locally representable} if there exists a sequence $(A_n)_{n\in\N}$ of unital finite dimensional
subalgebras of $A$ such that
\bi\item $\bigcup\limits_{n\in\N} A_n$ is dense in $A$
\item $\alpha_g(A_n)\subseteq A_n$ for all $g\in G$ and all $n$ in $\N$.
\item $\alpha|_{A_n}$ is inner for all $n$ in $\N$. \ei
\end{df}

Product type actions on UHF-algebras are examples of locally representable actions. Such actions have been
classified in terms of their equivariant $K$-theory by Handelman and Rossmann in \cite{handelman rossmann}.\\
\indent Using Theorem \ref{K-thy restrictions for K-freeness} and a result from \cite{phillips equivariant k-theory book},
we are able to describe all locally representable actions $\alpha$ of
a compact Lie group $G$ on an AF-algebra with $\cdimRok(\alpha)<\I$: the group $G$ must be finite, and all such actions are
conjugate to a specific model action $\mu^G$, so in particular they all have the Rokhlin property.
The model action $\mu^G\colon G\to \Aut(M_{|G|^\infty})$ is the infinite tensor product of copies of the left regular representation.
(We identify $M_{|G|}$ with $\K(\ell^2(G))$ in the usual way.) It is well known that $\mu^G$
(and any tensor product of it with any other action) has the Rokhlin property; see \cite{izumi finite group actions I}.

\begin{cor}\label{cor: loc rep AF-action} Let $G$ be a compact Lie group, let $A$ be a unital AF-algebra, and let $\alpha
\colon G\to\Aut(A)$ be a locally representable AF-action. Then \tfae:
\be\item $\alpha$ has the \Rp;
\item $\alpha$ has finite \Rdim\ with commuting towers;
\item $\alpha$ is totally $K$-free;
\item $\alpha$ has discrete $K$-theory. \ee
Moreover, if any of the above holds, then $G$ must be finite and there is an equivariant isomorphism 
$$(A,\alpha)\cong (A\otimes M_{|G|^\I},\id_A\otimes \mu^G).$$
In particular, $\alpha$ absorbs $\mu^G$ tensorially. \end{cor}
\begin{proof} The implications (1) $\Rightarrow$ (2) $\Rightarrow$ (3) are true
in general. The equivalence between (3) and (4) follows from Theorem 9.2.4 of
\cite{phillips equivariant k-theory book}. In particular, any of the conditions (1) through (4) implies that $\alpha$ is
totally $K$-free, so $G$ must be finite by Theorem \ref{K-thy restrictions for K-freeness}. Now, the fact that the
second condition implies the fifth in Theorem 9.2.4 of \cite{phillips equivariant k-theory book}, which in turn follows
from the classification results in \cite{handelman rossmann}, shows that (3) implies the existence of
an equivariant isomorphism
$$(A,\alpha)\cong (A\otimes M_{|G|^\I},\id_A\otimes \mu^G).$$
We conclude that $\alpha$ has the Rokhlin property, so (3) implies (1).\\
\indent The final claim follows from the fact that $\mu^G$ absorbs itself tensorially.
%there is an equivariant isomorphism 
%$(M_{|G|^\I}\otimes M_{|G|^\I},\mu^G\otimes \mu^G)\cong (M_{|G|^\I},\mu^G)$.
\end{proof}

We close this section by summarizing the known implications between the notions we have studied in this section.

\begin{thm}\label{thm: summary}
Let $G$ be a compact Lie group, let $A$ be a \uca\ and let $\alpha\colon G\to\Aut(A)$ be a continuous action.
Consider the following conditions for the action $\alpha$:
\be
\item[$(a)$] Rokhlin property.
\item[$(b)$] Finite Rokhlin dimension with commuting towers, $\cdimRok(\alpha)<\I$.
\item[$(c)$] $X$-Rokhlin property for some free $G$-space $X$.
\item[$(d)$] Finite Rokhlin dimension, $\dimRok(\alpha)<\I$.
\item[$(e)$] Pointwise outerness.
\item[$(f)$] Discrete $K$-theory.
\item[$(g)$] Total $K$-freeness.
\ee

We have the following implications, where a theorem or corollary is referenced when it proves the implication in question:

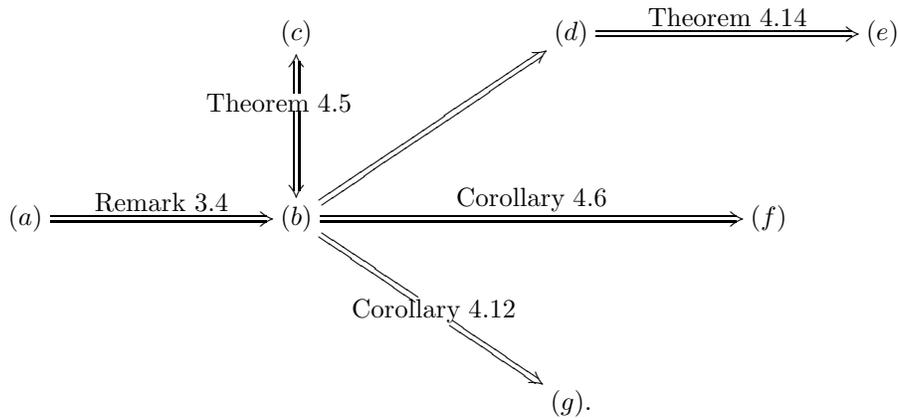
\beqa\xymatrix{ &&& (c) &&& (d) \ar@{=>}[rrr]^-{\mbox{Theorem \ref{pointwise outer}}} &&& (e) \\
 &&&&&&&&& \\
(a)\ar@{=>}[rrr]^-{\mbox{Remark \ref{rem: Rp and Rdim}}} &&&
(b)\ar@{<=>}[uu]|>>>>>>>{\mbox{\mbox{Theorem \ref{strongly free and finite Rdim} \ \ \ }}}\ar@{=>}[uurrr]
\ar@{=>}[rrrrr]^-{\mbox{Corollary \ref{cor: discrete K-thy}}}\ar@{=>}[ddrrr]|{\mbox{Corollary \ref{cor: totKfree}}} &&&&& (f) & \\
 &&&&&&&&& \\
&&& &&& (g). &&& }\eeqa

None the above arrows can be reversed in full generality, and presumably there are no other implications between the stated
conditions. In the diagram below, a dotted arrow means that the implication does not hold in general, and in
each case a counterexample is referenced:

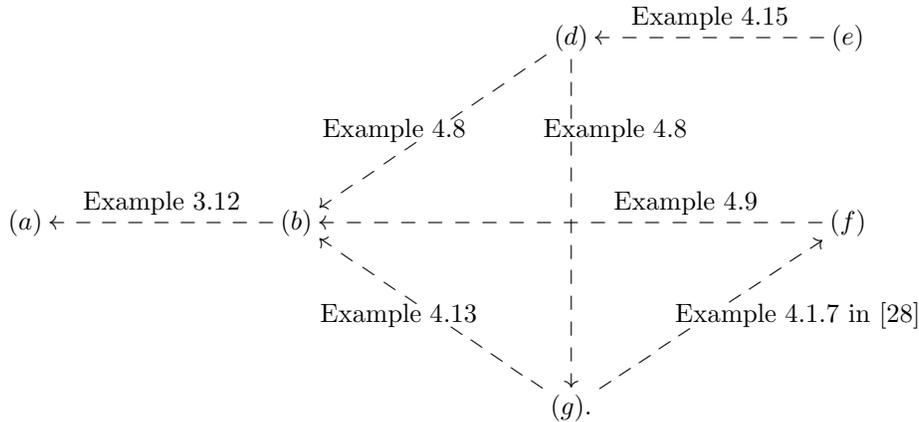
\beqa\xymatrix{ &&& &&& (d)\ar@{-->}[dddd]|<<<<<<<<<<{ \ \ \ \ \ \ \ \ \ \ \mbox{Example \ref{eg: commuting matters}} }
\ar@{-->}[ddlll]|{\mbox{Example \ref{eg: commuting matters} \ \ \ \ \ \ \ \ }} &&& (e)\ar@{-->}[lll]_-{\mbox{Example \ref{eg: pointwise outer not Rdim}}} \\
 &&&&&&&&& \\
(a)&&&
(b)\ar@{-->}[lll]_-{\mbox{Example \ref{eg: Rdim not Rp}}} &&&&&& (f)\ar@{-->}[llllll]_-{ \ \ \ \ \ \ \ \ \ \ \ \ \ \ \ \ \ \ \ \ \ \ \ \ \ \
\mbox{Example \ref{eg: discr K-thy not Rdim}}} \\
 &&&&&&&&& \\
&&& &&& (g)\ar@{-->}[uulll]|{\mbox{Example \ref{eg: tot K-free not Rdim}} \ \ \ \ \ \ \ \ }\ar@{-->}[uurrr]|{ \ \ \ \ \ \ \ \ \ \ \ \ \ \ \ \ \ \ \ \ \mbox{Example 4.1.7 in \cite{phillips equivariant k-theory book}}}. &&& }\eeqa

Finally, some of the arrows in the first diagram can be reversed in special situations:
\be\item If $A$ is commutative, then conditions $(b), (c), (d), (f)$ and $(g)$ are all equivalent to each other, and equivalent to freeness of the action
on the maximal ideal space, by Theorem \ref{strongly free and finite Rdim} and Atiyah-Segal completion theorem (see, for example, Theorem 1.1.1
in \cite{phillips equivariant k-theory book}). Condition $(e)$ is not equivalent to the others by Example \ref{eg: pointwise outer not Rdim}, and
neither is condition $(a)$ by Example \ref{eg: Rdim not Rp}.
\item If $A$ is an AF-algebra and $\alpha$ is a locally representable AF-action (see Definition \ref{df: loc repr}), then conditions $(a), (b),
(c), (f)$ and $(g)$ are equivalent by Corollary \ref{cor: loc rep AF-action}.
\item If $A$ is a Kirchberg algebra and $G=\Z_2$ (and possibly also if $G$ is any finite group), then $(e)$ and $(f)$ are
equivalent by Theorem 2.3 in \cite{BEMSW}. \ee
\end{thm}

\section{Outlook and open problems}

In this last section, we give some indication of possible directions for future work, and raise some natural questions related
to our findings. \\
\indent Although some of our results, particularly in Section 4, assume that the acting group is a Lie group, this is
probably not needed everywhere. Our first suggested problem is then:

\begin{pbm} Extend some of the results in this paper to actions of not necessarily finite-dimensional compact groups.\end{pbm}

We point out that the assumption that $G$ be a Lie group in Corollary 4.3 in \cite{HP} is necessary, since it relies on
Atiyah-Segal completion Theorem (see \cite{AS}, and see Theorem 1.1.1 in \cite{phillips equivariant k-theory book}),
for which one needs the representation ring $R(G)$ to be finitely generated. We suspect that Corollary \ref{cor: discrete K-thy}
is not true in general for arbitrary compact groups (it probably already fails for actions on compact spaces), but it may be the case
that all compact group actions with finite Rokhlin dimension with commuting towers have \emph{locally} discrete $K$-theory.\\
\indent Somewhat related, we ask:

\begin{qst}\label{qst: X-Rp} 
Is finite Rokhlin dimension with commuting towers equivalent to the $X$-Rokhlin property 
for \emph{arbitrary} compact groups?\end{qst}

Maybe one should start with the commutative case:
 
\begin{qst}\label{qst: freeness}
Is finite Rokhlin dimension for actions on commutative \ca s, equivalent to freeness of the induced action on the maximal ideal space for \emph{arbitrary} compact groups?\end{qst}

For compact Lie groups, the answer is yes in both cases; see Theorem \ref{strongly free and finite Rdim} and
Theorem \ref{free action on spaces}. It was crucial in the
proofs of said theorems that free actions of compact Lie groups have local cross-sections. We suspect that the answer
to Question \ref{qst: X-Rp} and Question \ref{qst: freeness} is `no', and it is possible that the action of $G=\prod\limits_{n\in\N}\Z_n$ on
$X=\prod\limits_{n\in\N}S^1$ by coordinate-wise rotation is a counterexample (to both questions). Such action has the $X$-Rokhlin property and is free, but the quotient map $X\to X/G$ is known not to have local cross-sections, 
since $X$ is not locally homeomorphic to $X/G\times G$. We have not checked, however, whether this action
has finite Rokhlin dimension.\\
\ \\
\indent The results in \cite{BEMSW} show that for $\Z_2$-actions on Kirchberg algebras, pointwise outerness implies Rokhlin
dimension at most 1 with noncommuting towers. It is conceivable that a similar result holds for a larger class of finite groups,
and presumably all of them.

\begin{qst} Is pointwise outerness equivalent to finite Rokhlin dimension (with noncommuting towers) for finite group actions
on Kirchberg algebras? \end{qst}

If the question above has an affirmative answer, as the results in \cite{BEMSW} suggest, one may try to prove the corresponding result for simple unital separable nuclear
\ca s with tracial rank zero, where presumably some additional assumptions will be needed.\\
\indent Alternatively,

\begin{pbm} Find obstructions (not necessarily $K$-theoretical) for having an action of a finite (or compact) group with finite
Rokhlin dimension with noncommuting towers. \end{pbm}

The results in Section 4 of \cite{gardella classification on Otwo} suggest the following:

\begin{cnj} Let $G$ be a compact Lie group and let $\alpha\colon G\to\Aut(\mathcal{O}_2)$ be an action with finite Rokhlin dimension with commuting towers. Then $\alpha$ has the Rokhlin property. \end{cnj}

Example \ref{eg: commuting matters} shows that the corresponding statement for noncommuting towers is in general false.\\
\indent Based on Corollary \ref{cor: discrete K-thy} and the comments and examples after it, we ask:

\begin{qst} If $\alpha\colon G\to\Aut(A)$ is an action of a compact Lie group on a unital \ca\ $A$\ and $\cdimRok(\alpha)<\I$,
does one have $$\min\left\{n\colon I_G^n\cdot K_\ast^G(A)=0\right\}\leq \cdimRok(\alpha)+1 ,$$
or any other relationship between these quantities?\end{qst}

Example \ref{eg: min does not determin dim} shows that one cannot in general expect equality to hold.\\
\ \\
\indent Finally, the following problem is likely to be challenging:

\begin{pbm} Can actions with finite Rokhlin dimension with commuting towers on unital Kirchberg algebras satisfying
the UCT be classified, in a way similar to what was done in \cite{izumi finite group actions II} for finite group actions with the
Rokhlin property, or in \cite{gardella Kirchberg 1} and \cite{gardella Kirchberg 2} for circle actions with the Rokhlin property?\end{pbm}

Some of these questions will be addressed in \cite{GHS}.

\end{document}